\documentclass[11pt,reqno,twoside]{amsart}
\usepackage{amssymb,amsmath,amsthm,soul,color,paralist}
\usepackage{t1enc}
\usepackage[cp1250]{inputenc}
\usepackage{a4,indentfirst,latexsym}
\usepackage{graphics}
\usepackage{mathrsfs}
\usepackage{cite,enumitem,graphicx}
\usepackage[colorlinks=true,urlcolor=blue,
citecolor=red,linkcolor=blue,linktocpage,pdfpagelabels,
bookmarksnumbered,bookmarksopen]{hyperref}
\usepackage[english]{babel}
\usepackage[left=2.50cm,right=2.50cm,top=2.72cm,bottom=2.72cm]{geometry}
\usepackage[metapost]{mfpic}
\usepackage[hyperpageref]{backref}
\usepackage[colorinlistoftodos]{todonotes}
\usepackage[normalem]{ulem}

\makeatletter
\providecommand\@dotsep{5}
\def\listtodoname{List of Todos}
\def\listoftodos{\@starttoc{tdo}\listtodoname}
\makeatother

\numberwithin{equation}{section}

 \usepackage[colorlinks=true]{hyperref}
\hypersetup{urlcolor=blue, citecolor=red}

\newcommand{\R}{\mathbb{R}}

\DeclareMathOperator{\X}{\mathbb{H}}

\DeclareMathOperator{\e}{\varepsilon}

\newcommand{\N}{\mathcal{N}}
\DeclareMathOperator{\J}{\mathcal{J}}

\newtheorem{theorem}{Theorem}[section]
\newtheorem{corollary}{Corollary}

\newtheorem{lemma}[theorem]{Lemma}
\newtheorem{proposition}{Proposition}

\theoremstyle{definition}
\newtheorem{definition}[theorem]{Definition}
\newtheorem{remark}{Remark}

\title[Critical fractional Schr\"odinger systems]{Concentration phenomena for critical fractional Schr\"odinger systems} 

\author[V. Ambrosio]{Vincenzo Ambrosio}
\address{
Vincenzo Ambrosio\hfill\break\indent 
Dipartimento di Scienze Pure e Applicate (DiSPeA),\hfill\break\indent
Universit\`a degli Studi di Urbino `Carlo Bo'\hfill\break\indent
Piazza della Repubblica, 13\hfill\break\indent
61029 Urbino (Pesaro e Urbino, Italy)}
\email{vincenzo.ambrosio@uniurb.it}

\subjclass{Primary: 35A15, 35J50; Secondary: 35R11, 58E05.}
\keywords{Fractional Schr\"odinger systems, variational methods, critical exponent.}

\begin{document}
\maketitle


\begin{abstract}
In this paper we study the existence, multiplicity and concentration behavior of solutions for the following critical fractional Schr\"odinger system
\begin{equation*}
\left\{
\begin{array}{ll}
\e^{2s} (-\Delta)^{s}u+V(x)u=Q_{u}(u, v)+\frac{1}{2^{*}_{s}}K_{u}(u, v) &\mbox{ in } \R^{N}\\
\e^{2s} (-\Delta)^{s}u+W(x)v=Q_{v}(u, v)+\frac{1}{2^{*}_{s}}K_{v}(u, v) &\mbox{ in } \R^{N} \\
u, v>0 &\mbox{ in } \R^{N},
\end{array}
\right.
\end{equation*}
where $\e>0$ is a parameter, $s\in (0, 1)$, $N>2s$,  $(-\Delta)^{s}$ is the fractional Laplacian operator, $V:\R^{N}\rightarrow \R$ and $W:\R^{N}\rightarrow \R$ are positive H\"older continuous potentials, $Q$ and $K$ are homogeneous $C^{2}$-functions having subcritical and critical growth respectively.\\
We relate the number of solutions with the topology of the set where the potentials $V$ and $W$ attain their minimum values. The proofs rely on the Ljusternik-Schnirelmann theory and variational methods.
\end{abstract}

\section{Introduction}
\noindent
During the last years there has been a renewed and increasing interest in the study of nonlocal diffusion problems, in particular to the ones driven by the fractional Laplace operator, not only for a pure academic interest, but also for the several applications in different fields, such as, among the others, the thin obstacle problem, optimization, finance, phase transitions, anomalous diffusion, crystal dislocation, conservation laws, ultra-relativistic limits of quantum mechanics and water waves. For an elementary introduction to this topic we refer the interested reader  to \cite{DPV, MBRS}.\\
In this paper we deal with the following class of nonlinear fractional Schr\"odinger systems 
\begin{equation}\label{P}
\left\{
\begin{array}{ll}
\e^{2s} (-\Delta)^{s}u+V(x)u=Q_{u}(u, v)+\frac{1}{2^{*}_{s}}K_{u}(u, v) &\mbox{ in } \R^{N}\\
\e^{2s} (-\Delta)^{s}u+W(x)v=Q_{v}(u, v)+\frac{1}{2^{*}_{s}}K_{v}(u, v) &\mbox{ in } \R^{N} \\
u, v>0 &\mbox{ in } \R^{N},
\end{array}
\right.
\end{equation}
where $\e>0$ is a parameter, $s\in (0, 1)$, $N>2s$, $2^{*}_{s}=\frac{2N}{N-2s}$ is the fractional critical Sobolev exponent, $Q$ and $K$ are homogeneous $C^{2}$-functions having subcritical and critical growth respectively.\\
The fractional Laplacian operator $(-\Delta)^{s}$ may be defined for any $u: \R^{N}\rightarrow \R$ belonging to the Schwartz space $\mathcal{S}(\R^{N})$ of rapidly decaying functions, by setting 
$$
(-\Delta)^{s}u(x)=-\frac{1}{2}C(N, s)\int_{\R^{N}} \frac{u(x+y)+u(x-y)-2u(x)}{|y|^{N+2s}} \,dy \quad \forall x\in \R^{N},
$$
where 
$$
C(N, s)=\left(\int_{\R^{N}} \frac{1-\cos x_{1}}{|x|^{N+2s}} \,dx\right)^{-1};
$$
see \cite{BucurV, DPV, Silvestre, Stein} for more details. 
From a probabilistic point of view, the fractional Laplacian may be viewed as the infinitesimal generator of L\'evy stable diffusion processes  \cite{App}.\\
Problem \eqref{P} arises in the study of the solitary wave solutions $\psi_{1}(t, x)=e^{-\frac{\imath c_{1} t}{\e}} u(x)$ and $\psi_{2}(t, x)=e^{-\frac{\imath c_{2} t}{\e}} v(x)$  of time-dependent coupled fractional nonlinear Schr\"odinger equations
\begin{equation}\label{TDSS}
\left\{
\begin{array}{ll}
\imath \frac{\partial \psi_{1}}{\partial t}= \e^{2s}(-\Delta)^{s}\psi_{1}+\tilde{V}(x)\psi_{1}-f(\psi_{1}, \psi_{2})  &x\in\R^{N}, t>0\\
\imath \frac{\partial \psi_{2}}{\partial t}= \e^{2s}(-\Delta)^{s}\psi_{2}+\tilde{W}(x)\psi_{1}-g(\psi_{1}, \psi_{2}) & x\in\R^{N}, t>0
\end{array}
\right.
\end{equation}
where $\tilde{V}$ and $\tilde{W}$ are external potentials, $f$ and $g$ are suitable nonlinearities, and $\e$ is a sufficiently small parameter which corresponds to the Planck constant.\\
We recall that the time-dependent fractional Schr\"odinger equation
\begin{equation*}
\imath \frac{\partial \psi}{\partial t}= \e^{2s}(-\Delta)^{s}\psi+\tilde{V}(x)\psi-f(\psi) \mbox{ for } x\in\R^{N}, t>0
\end{equation*}
was introduced by Laskin \cite{Laskin1, Laskin2} and it is a fundamental equation of fractional quantum mechanics in the study of particles on stochastic fields modeled by L\'evy processes. For more physical background we refer to \cite{A, DDPDPV, DDPW, FQT, Secchi}.
 
When $u=v$ and $K(u, u)=|u|^{2^{*}_{s}}$, the system \eqref{P} reduces to a fractional critical Schr\"odinger equation of the type
\begin{align}\label{CFSE}
 \e^{2s}(-\Delta)^{s} u+V(x)u=f(x, u)+|u|^{2^{*}_{s}-2}u \mbox{ in } \R^{N},
\end{align}
which is currently actively studied by many authors.\\
Shang and Zhang \cite{SZ} proved the existence of a nonnegative ground state solution to \eqref{CFSE} with $f(x, u)=\lambda f(u)$
and they investigated the relation between the number of solutions and the topology of the set where $V$ attains its minimum, for $\lambda$ sufficiently large and $\e$ small enough. 
Zhang et al. \cite{XZR} studied, via the principle of concentration compactness in the fractional Sobolev space and minimax arguments, the existence of nontrivial radially symmetric solutions to \eqref{CFSE} with $\e=1$, when $V(x)$ is radially symmetric, and $f(x, u)=k(x)f(u)$, where $k$ is a bounded radially symmetric function and $f$ has a subcritical growth.
Teng \cite{teng} obtained the existence of a nontrivial ground state for a nonlinear fractional Schr\"odinger-Poisson system with critical Sobolev exponent, by using the method of Nehari manifold, the monotonic trick and a global compactness Lemma. 
By applying variational methods and Ljusternik-Schnirelmann theory, He and Zou \cite{HZ} proved existence and multiplicity of solutions to \eqref{CFSE} under a local condition on the potential $V$.\\
Dipierro et al. \cite{DMPV} used the Lyapunov-Schmidt reduction to obtain some bifurcation results for \eqref{CFSE} with $f(x, u)=\e h(x) u^{q}$ with $q\in (0, 2^{*}_{s}-1)$ and $h$ is a continuous and compactly supported function.
Fiscella and Pucci \cite{FP} dealt with Kirchhoff  type equations, driven by the $p$-fractional Laplace operator, involving critical Hardy-Sobolev nonlinearities and nonnegative potentials.\\
Motivated by the above papers, in this work we focus our attention on  the multiplicity and the concentration of solutions to the critical fractional system \eqref{P}. 
We recall that a solution $(u_{\e}, v_{\e})$ to \eqref{P} concentrates at some point $x_{0}\in \R^{N}$ as $\e\rightarrow 0$ provided 
\begin{align*}
\forall \delta>0 \quad \exists \e_{0}, R>0 \, : \,  |(u_{\e}(x), v_{\e}(x))|\leq \delta, \quad \forall \, |x-x_{0}|\geq \e R, \, \e\in (0, \e_{0}). 
\end{align*}
By using suitable variational methods, we will relate the number of solutions to \eqref{P} with the topology of the set where the potentials $V$ and $W$ attain their minimum values. 
In order to achieve our aim, along the paper we will suppose that the potentials $V:\R^{N}\rightarrow \R$ and $W:\R^{N}\rightarrow \R$ are H\"older continuous functions, and there exist $\Lambda\subset \R^{N}$, $x_{0}\in \R^{N}$ and $\rho_{0}>0$ such that:
\begin{compactenum}[$(H1)$]
\item $V(x), W(x)\geq \rho_{0}$ for any $x\in \partial \Lambda$;
\item $V(x_{0}), W(x_{0})< \rho_{0}$;
\item $V(x)\geq V(x_{0})>0$, $W(x)\geq W(x_{0})>0$ for any $x\in \R^{N}$.
\end{compactenum}
Now, we state the assumptions on the functions $Q(u, v)$ and $K(u, v)$.
We assume that $Q\in C^{2}(\R^{2}_{+}, \R)$, where $\R^{2}_{+}=[0, \infty)\times [0, \infty)$, verifies the following conditions:
\begin{compactenum}[$(Q1)$]
\item there exists $p\in (2, 2^{*}_{s})$ such that $Q(t u, tv)=t^{p}Q(u, v)$ for all $t>0$, $(u, v)\in \R^{2}_{+}$;
\item there exists $C>0$ such that $|Q_{u}(u, v)|+|Q_{v}(u, v)|\leq C(u^{p-1}+v^{p-1})$ for all $(u, v)\in \R^{2}_{+}$;
\item $Q_{u}(0, 1)=0=Q_{v}(1, 0)$;
\item $Q_{u}(1, 0)=0=Q_{v}(0, 1)$;
\item $Q(u, v)>0$ for any $u, v>0$;
\item $Q_{u}(u, v), Q_{v}(u, v)\geq 0$ for all $(u, v)\in \R^{2}_{+}$.
\end{compactenum}
Regarding the function $K\in C^{2}(\R^{2}_{+}, \R)$, we make the following hypotheses:
\begin{compactenum}[$(K1)$]
\item  $K(t u, tv)=t^{2^{*}_{s}}K(u, v)$ for all $t>0$, $(u, v)\in \R^{2}_{+}$;
\item the $1$-homogeneous function $G: \R^{2}_{+}\rightarrow \R$ given by $G(u^{2^{*}_{s}}, v^{2^{*}_{s}})=K(u, v)$ is concave;
\item there exists $c>0$ such that $|K_{u}(u, v)|+|K_{v}(u, v)|\leq c(u^{2^{*}_{s}-1}+v^{2^{*}_{s}-1})$ for all $(u, v)\in \R^{2}_{+}$;
\item $K_{u}(0, 1)=0=K_{v}(1, 0)$;
\item $K_{u}(1, 0)=0=K_{v}(0, 1)$;
\item $K(u, v)>0$ for any $u, v>0$;
\item $K_{u}(u, v), K_{v}(u, v)\geq 0$ for all $(u, v)\in \R^{2}_{+}$.
\end{compactenum}
Since we are interested in positive solutions $(u, v)$ of \eqref{P}, that is $u, v>0$ in $\R^{N}$, we extend the functions $Q$ and $K$ to the whole $\R^{2}$ by setting $Q(u, v)=K(u, v)=0$ if $u\leq 0$ or $v\leq 0$. 
\begin{remark}
If $F(u, v)\in C^{2}$ is a $q$-homogeneous function, then we have the following identities:
\begin{equation}\label{2.1}
qF(u, v)=uF_{u}(u, v)+vF_{v}(u, v) \mbox{ for any } (u, v)\in \R^{2},
\end{equation}
and
\begin{equation}\label{2.11}
q(q-1)F(u, v)=u^{2}F_{uu}(u, v)+2uvF_{uv}(u, v)+v^{2}F_{vv}(u, v) \mbox{ for any } (u, v)\in \R^{2}.
\end{equation}
\end{remark}
\noindent
Now, we give some examples of functions $Q$ and $K$ satisfying our assumptions. \\
Let $q\geq 1$ and 
$$
\mathcal{P}_{q}(u, v)=\sum_{\alpha_{i}+\beta_{i}=q} a_{i} u^{\alpha_{i}} v^{\beta_{i}}
$$
where $i\in \{1, \dots, k\}$, $\alpha_{i}, \beta_{i}\geq 1$ and $a_{i}\in \R$. The following functions and their possible combinations, with appropriate choice of the coefficients $a_{i}$, verify our assumptions  on $Q$
$$
Q_{1}(u, v)=\mathcal{P}_{p}(u, v),  \quad Q_{2}(u, v))=\sqrt[r]{\mathcal{P}_{\ell}(u, v)} \quad \mbox{ and } \quad Q_{3}(u, v)=\frac{\mathcal{P}_{\ell_{1}}(u, v)}{\mathcal{P}_{\ell_{2}}(u, v)},
$$
with $r= \ell p$ and $\ell_{1}-\ell_{2}=p$. As a model for $K$, we can take $K(u, v)=\mathcal{P}_{2^{*}_{s}}(u, v)$.

We would like to note that when $s=1$,  the above hypotheses on $Q$ and $K$ have been used in \cite{AFF2} (see also \cite{DMFS}) to study the concentration phenomena of solutions for the following elliptic system
\begin{equation}\label{ECS}
\left\{
\begin{array}{ll}
 -\e^{2}\Delta u+V(x)u=Q_{u}(u, v)+\frac{1}{2^{*}_{s}}K_{u}(u, v) &\mbox{ in } \R^{N}\\
 -\e^{2}\Delta v+W(x)v=Q_{v}(u, v)+\frac{1}{2^{*}_{s}}K_{v}(u, v) &\mbox{ in } \R^{N} \\
u, v>0 &\mbox{ in } \R^{N};
\end{array}
\right.
\end{equation}
further interesting results for elliptic systems can be found in \cite{AS, AY, BS, FF, Hichem, WA}.\\
In this paper, we extend the existence and multiplicity results obtained in \cite{AFF2} to the nonlocal framework.\\
It is worth observing that nowadays there are several papers dealing with fractional systems set in bounded domain and in the whole $\R^{N}$ \cite{ ASy, A5, choi, DPP, FMPS, glz, lm, TVZ, ww}, but, to our knowledge, no results on the multiplicity and concentration of solutions for critical fractional Schr\"odinger systems have been established. The purpose of this work is to give a first result in this direction.\\ 
Before stating our main result, we introduce some useful notations. 
Let $\xi\in \R^{N}$ fixed, and we consider the following autonomous system
\begin{equation}\label{P0}
\left\{
\begin{array}{ll}
 (-\Delta)^{s}u+V(\xi)u=Q_{u}(u, v)+\frac{1}{2^{*}_{s}}K_{u}(u, v)  &\mbox{ in } \R^{N}\\
 (-\Delta)^{s}u+W(\xi)v=Q_{v}(u, v)+\frac{1}{2^{*}_{s}}K_{v}(u, v) &\mbox{ in } \R^{N} \\
u, v>0 &\mbox{ in } \R^{N}.
\end{array}
\right.
\end{equation}
We set $\X_{0}=H^{s}(\R^{N})\times H^{s}(\R^{N})$ endowed with the following norm
$$
\|(u, v)\|^{2}_{\xi}=\int_{\R^{N}} |(-\Delta)^{\frac{s}{2}} u|^{2}+ |(-\Delta)^{\frac{s}{2}} v|^{2} dx+\int_{\R^{N}} (V(\xi) u^{2}+W(\xi) v^{2}) dx. 
$$
Let $\J_{\xi}: \X_{0}\rightarrow \R$ be the Euler-Lagrange functional associated to the above problem, i.e. 
$$
\J_{\xi}(u, v)=\frac{1}{2}\|(u, v)\|^{2}_{\xi}-\int_{\R^{N}} Q(u, v) \,dx-\frac{1}{2^{*}_{s}} \int_{\R^{N}} K(u, v)  \, dx. 
$$
From the assumptions $(H3)$, $(Q1)$ and $(Q2)$, it is easy to see that $\J_{\xi}$ possesses a mountain pass geometry, so we can consider the mountain pass value
$$
C(\xi)=\inf_{\gamma\in \Gamma} \max_{t\in [0, 1]} \J_{\xi}(\gamma(t))
$$
where 
$$
\Gamma=\{\gamma\in C([0, 1], \X_{0}): \gamma(0)=0, \J_{\xi}(\gamma(1))\leq 0\}.
$$
Moreover, $C(\xi)$ can be also characterized as
$$
C(\xi)=\inf_{(u, v)\in \mathcal{N}_{\xi}} \J_{\xi}(u, v),
$$
where $\mathcal{N}_{\xi}$ is the Nehari manifold associated of $\J_{\xi}$. 
As proved in Section $3$, for any fixed $\xi\in \R^{N}$, $C(\xi)$ is achieved, so that
$$
M:=\left\{x\in \R^{N}: C(x)=\inf_{\xi\in \R^{N}} C(\xi)\right\}\neq \emptyset.
$$
Arguing as in \cite{A5}, we can prove that $\xi \mapsto C(\xi)$ is a continuous function and
$$
C^{*}=C(x_{0})=\inf_{\xi\in \Lambda} C(\xi)<\min_{\xi\in \partial \Lambda} C(\xi).
$$
We recall that if $Y$ is a given closed set of a topological space $X$, we denote by $cat_{X}(Y)$ the Ljusternik-Schnirelmann category of $Y$ in $X$, that is the least number of closed and contractible sets in $X$ which cover $Y$.
\smallskip

\begin{theorem}\label{thm1}
Assume that $(H1)$-$(H3)$ hold, and $Q$ and $K$ verify $(Q1)$-$(Q7)$and $(K1)$-$(K7)$ respectively.
In addition, we make the following technical assumption on $Q$:
\begin{compactenum}[$(Q7)$]
\item $Q(u, v)\geq \lambda u^{\tilde{\alpha}} v^{\tilde{\beta}}$ for any $(u, v)\in \R_{+}^{2}$ with $1<\tilde{\alpha}, \tilde{\beta}<2^{*}_{s}$, $\tilde{\alpha}+\tilde{\beta}=q_{1}\in (2, 2^{*}_{s})$, and $\lambda$ verifying 
\begin{itemize}
\item $\lambda>0$ if either $N\geq 4s$, or $2s<N<4s$ and $2^{*}_{s}-2<q_{1}<2^{*}_{s}$;
\item $\lambda$ is sufficiently large if  $2s<N<4s$ and $2<q_{1}\leq 2^{*}_{s}-2$.
\end{itemize}
\end{compactenum}
Then, for any $\delta>0$ satisfying
$$
M_{\delta}=\{x\in \R^{N}: dist(x, M)<\delta\}\subset \Lambda,
$$ 
there exists $\e_{\delta}>0$ such that for any $\e\in (0, \e_{\delta})$, the system \eqref{P} admits at least $cat_{M_{\delta}}(M)$ positive solutions.
Moreover, if $(u_{\e}, v_{\e})$ is a solution to \eqref{P} and $P_{\e}$ and $Q_{\e}$ are maximum points of $u_{\e}$ and $v_{\e}$ respectively, then $C(P_{\e}), C(Q_{\e})\rightarrow C(x_{0})$ as $\e\rightarrow 0$, and we have the following estimates 
\begin{align}\label{DEuv}
u_{\e}(x)\leq \frac{C \e^{N+2s}}{\e^{N+2s}+|x-P_{\e}|^{N+2s}} \,\, \mbox{ and } \,\, v_{\e}(x)\leq \frac{C \e^{N+2s}}{\e^{N+2s}+|x-Q_{\e}|^{N+2s}}. 
\end{align}
\end{theorem}

\noindent
We note that Theorem \ref{thm1} represents the nonlocal counterpart of Theorem $1.1$ proved  in \cite{AFF2}.\\
One of the main difficulties of the analysis of the problem \eqref{P} is due to the nonlocal character of the fractional Laplacian. To circumvent this hitch, Caffarelli and Silvestre \cite{CS} proved 
that one can localize the operator $(-\Delta)^{s}$ by considering it as the Dirichlet to Neumann operator associated to the $s$-harmonic extension in the halfspace $\R^{N+1}_{+}$, paying the price to add a new variable.
Anyway, in this work, we prefer to analyze the problem directly in $H^{s}(\R^{N})\times H^{s}(\R^{N})$, in order to adapt in our context some ideas developed in local setting (i.e. $s=1$) in \cite{Alves, AFF, AFF2, DMFS}.\\
The proof of Theorem \ref{thm1} is variational and it is based on the approach developed in \cite{A5} to study subcritical fractional systems.
Since we do not have any information about the behavior of the potential $V$ at the infinity, we are not able to show that the functional associated to \eqref{P} satisfies any compactness condition.
For this reason, as in \cite{Alves, DF}, we introduce a suitable penalization function modifying the nonlinearity $Q(u, v)+\frac{1}{2^{*}_{s}}K(u, v)$ outside $\Lambda$. Then, we can prove that the associated modified functional satisfies the Palais-Smale condition at every level $c<\frac{s}{N}\widetilde{S}_{K}$, where the number
\begin{equation*}
\widetilde{S}_{K}=\inf_{(u, v)\in \X_{0}\setminus\{(0, 0)\}} \frac{\int_{\R^{N}} |(-\Delta)^{\frac{s}{2}} u|^{2}+|(-\Delta)^{\frac{s}{2}} v|^{2} dx}{\left(\int_{\R^{N}} K(u^{+}, v^{+}) dx\right)^{\frac{2}{2^{*}_{s}}}}
\end{equation*}
is strongly related to the best constant $S_{*}$ of the Sobolev embedding $H^{s}(\R^{N})$ into $L^{2^{*}_{s}}(\R^{N})$. As observed in  \cite{ASy}, $\widetilde{S}_{K}$ plays a fundamental role when we have to study critical systems like \eqref{P}.
Clearly, due to the presence of the critical Sobolev exponent, the calculations needed to get compactness for the critical modified functional are more complicated than the ones performed in the subcritical case.
After that, by using the technique introduced  by Benci and Cerami \cite{BC} and Ljusternik-Schnirelmann theory, we obtain multiple solutions of the critical modified problem. 
It remains to prove that the solutions $(u_{\e}, v_{\e})$ of the critical modified problem are indeed solutions to \eqref{P} for $\e>0$  sufficiently small. Unfortunately, the methods used in \cite{Alves, AFF, AFF2} do not work in our context due to the nonlocality of $(-\Delta)^{s}$. To overcome this difficulty, we take care of the arguments employed in \cite{AM, A3, HZ} to study scalar fractional Schr\"odinger equations.
More precisely, by using some properties of the Bessel kernel \cite{FQT}, we can see that the sum $u_{\e}+v_{\e}$ of the solutions of \eqref{P} has a power-type decay at infinity, and this will be fundamental  to achieve our aim.\\
We would like to note that  Theorem \ref{thm1} complements the results obtained in \cite{AM, A3, A5, HZ}, in the sense that we are considering the multiplicity for fractional critical systems. 
\smallskip

\noindent
The structure of the paper is the following. In Section $2$ we recall some useful facts about the fractional Sobolev spaces.  
In Section $3$ we deal with the fractional autonomous systems associated to \eqref{P}. In Section $4$
we introduce the modified problem and we provide some fundamental compactness results. In Section $5$ we prove that the modified problem admits multiple solutions.
In Section $6$ we give the proof of Theorem \ref{thm1}.

\section{Fractional Sobolev spaces}
\noindent
In this section we offer a rather sketchy review of the fractional Sobolev spaces and some useful results which will be used later. For more details, we refer to \cite{DPV, MBRS}.

Fixed $s\in (0,1)$, we denote by $\dot{H}^{s}(\R^{N})$ the set of functions $u\in L^{2^{*}_{s}}(\R^{N})$ such that  $\int_{\R^{N}} |(-\Delta)^{\frac{s}{2}} u|^{2} \,dx<\infty$, where $2^{*}_{s}=\frac{2N}{N-2s}$ is the fractional critical Sobolev exponent.
We define the fractional Sobolev space
$$
H^{s}(\R^{N})= \left\{u\in L^{2}(\R^{N}) : \int_{\R^{N}} |(-\Delta)^{\frac{s}{2}} u|^{2} \,dx<\infty \right \}
$$
endowed with the natural norm 
$$
\|u\|_{H^{s}(\R^{N})} = \sqrt{\int_{\R^{N}} |(-\Delta)^{\frac{s}{2}} u|^{2} dx + \int_{\R^{N}} |u|^{2} dx}.
$$

\noindent
We recall the following embeddings of the fractional Sobolev spaces into Lebesgue spaces.
\begin{theorem}\cite{DPV}\label{Sembedding}
Let $s\in (0,1)$ and $N>2s$. Then there exists a sharp constant $S_{*}=S(N, s)>0$
such that for any $u\in H^{s}(\R^{N})$
\begin{equation}\label{FSI}
\left(\int_{\R^{N}} |u|^{2^{*}_{s}} dx\right)^{\frac{2}{2^{*}_{s}}}  \leq S_{*} \int_{\R^{N}} |(-\Delta)^{\frac{s}{2}} u|^{2} dx . 
\end{equation}
Moreover $H^{s}(\R^{N})$ is continuously embedded in $L^{q}(\R^{N})$ for any $q\in [2, 2^{*}_{s}]$ and compactly in $L^{q}_{loc}(\R^{N})$ for any $q\in [2, 2^{*}_{s})$. 
\end{theorem}

\noindent
The following lemma is a version of the well-known concentration-compactness principle:
\begin{lemma}\cite{Secchi}\label{lionslemma}
Let $N>2s$. If $(u_{n})$ is a bounded sequence in $H^{s}(\R^{N})$ and if
$$
\lim_{n \rightarrow \infty} \sup_{y\in \R^{N}} \int_{B_{R}(y)} |u_{n}|^{2} dx=0
$$
where $R>0$,
then $u_{n}\rightarrow 0$ in $L^{t}(\R^{N})$ for all $t\in (2, 2^{*}_{s})$.
\end{lemma}

\noindent
Now, we define the quantity
$$
\widetilde{S}_{K}=\inf\left\{ \int_{\R^{N}} |(-\Delta)^{\frac{s}{2}}u|^{2}+|(-\Delta)^{\frac{s}{2}}v|^{2}\, dx: u, v\in H^{s}(\R^{N}),\,\int_{\R^{N}} K(u^{+}, v^{+})\, dx=1\right\}.
$$

\noindent
In the next Lemma, we prove an interesting relation between $S_{*}$ and $\widetilde{S}_{K}$.
\begin{lemma}\label{FSthm}
The following identity holds
\begin{equation}
\tilde{S}_{K}=S_{*} \frac{u_{0}^{2}+v^{2}_{0}}{[K(u_{0}, v_{0})]^{\frac{2}{2^{*}_{s}}}},
\end{equation}
where $(u_{0}, v_{0})\in \R^{2}$ is such that
$$
K(u_{0}, v_{0})^{\frac{2}{2^{*}_{s}}}=\max\{K(u, v)^{\frac{2}{2^{*}_{s}}}: |u|^{2}+|v|^{2}=1\}.
$$
We observe that the maximum point $(u_{0}, v_{0})$ exists in view of the fact that $K(u, v)^{\frac{2}{2^{*}_{s}}}$ is continuous and $\{(u, v)\in \R^{2}: |u|^{2}+|v|^{2}=1\}$ is compact.
\end{lemma}
\begin{proof}
Let $\{w_{n}\}\subset H^{s}(\R^{N})$ be a minimizing sequence for $S_{*}$, and we consider the sequence $\{(u_{0}w_{n}, v_{0}w_{n})\}$.
Then, by using the definition of $\tilde{S}_{K}$ and $(u_{0}, v_{0})$, and $(K1)$, we can see that
\begin{align*}
\tilde{S}_{K}&\leq \frac{\int_{\R^{N}} |(-\Delta)^{\frac{s}{2}} (u_{0} w_{n})|^{2}+|(-\Delta)^{\frac{s}{2}} (v_{0} w_{n})|^{2} \, dx}{(\int_{\R^{N}} K(u_{0} w^{+}_{n}, v_{0} w_{n}^{+})\, dx)^{\frac{2}{2^{*}_{s}}}} \\
&=\frac{u_{0}^{2}+v^{2}_{0}}{[K(u_{0}, v_{0})]^{\frac{2}{2^{*}_{s}}}} \frac{\int_{\R^{N}} |(-\Delta)^{\frac{s}{2}} w_{n}|^{2} \, dx}{(\int_{\R^{N}} |w_{n}|^{2^{*}_{s}}\, dx)^{\frac{2}{2^{*}_{s}}}}.
\end{align*}
Taking the limit as $n\rightarrow \infty$, we get 
\begin{equation}\label{FS1}
\tilde{S}_{K}\leq S_{*} \frac{u_{0}^{2}+v^{2}_{0}}{[K(u_{0}, v_{0})]^{\frac{2}{2^{*}_{s}}}}.
\end{equation}
Now, let $\{(u_{n}, v_{n})\}$ be a minimizing sequence for $\tilde{S}_{K}$. 
Recalling the definitions of $S_{*}$ and $(u_{0}, v_{0})$, and by using $(K1)$-$(K2)$, we have
\begin{align}\label{franci}
\frac{\int_{\R^{N}} |(-\Delta)^{\frac{s}{2}} u_{n}|^{2}+|(-\Delta)^{\frac{s}{2}} v_{n}|^{2} \, dx}{(\int_{\R^{N}} K(u^{+}_{n}, v_{n}^{+})\, dx)^{\frac{2}{2^{*}_{s}}}}&\geq S_{*} \frac{\|u_{n}\|^{2}_{L^{2^{*}_{s}}(\R^{N})}+\|v_{n}\|^{2}_{L^{2^{*}_{s}}(\R^{N})}}{K(\|u_{n}\|_{L^{2^{*}_{s}}(\R^{N})}, \|v_{n}\|_{L^{2^{*}_{s}}(\R^{N})})^{\frac{2}{2^{*}_{s}}}}\nonumber\\
&\geq S_{*} \frac{u_{0}^{2}+v^{2}_{0}}{[K(u_{0}, v_{0})]^{\frac{2}{2^{*}_{s}}}},
\end{align}  
where in the first inequality we used the following property for homogeneous function (see Proposition $4$ in \cite{DMFS}):
if $F$ a is $q$-homogeneous continuous function, with $q\geq 1$, and
the $1$-homogeneous function $G$ defined by
$$
G(u^{q}, v^{q})=F(u, v) \mbox{ for all } u, v\geq 0
$$
is concave,
then it holds the following H\"older type inequality 
\begin{equation*}
\int_{\R^{N}} F(u, v) \, dx\leq F(\|u\|_{L^{q}(\R^{N})}, \|v\|_{L^{q}(\R^{N})}) \mbox{ for all } u, v\in L^{q}(\R^{N}): u, v\geq 0.
\end{equation*}
Thus, by passing to the limit in \eqref{franci} as $n\rightarrow \infty$ we deduce that
\begin{equation}\label{FS2}
\tilde{S}_{K}\geq S_{*} \frac{u_{0}^{2}+v^{2}_{0}}{[K(u_{0}, v_{0})]^{\frac{2}{2^{*}_{s}}}}.
\end{equation}
Putting together \eqref{FS1} and \eqref{FS2} we have the desired result.
\end{proof}

\section{Autonomous critical system}
In this section we deal with the existence of solutions of the autonomous system associated to \eqref{P}.
Fixed $\xi\in \R^{N}$, we consider the following critical autonomous system
\begin{equation}\label{P0}
\left\{
\begin{array}{ll}
 (-\Delta)^{s}u+V(\xi)u=Q_{u}(u, v)+\frac{1}{2^{*}_{s}}K_{u}(u, v)  &\mbox{ in } \R^{N}\\
 (-\Delta)^{s}u+W(\xi)v=Q_{v}(u, v)+\frac{1}{2^{*}_{s}}K_{v}(u, v) &\mbox{ in } \R^{N} \\
u, v>0 &\mbox{ in } \R^{N}.
\end{array}
\right.
\end{equation}
Let $\X_{0}=H^{s}(\R^{N})\times H^{s}(\R^{N})$ endowed with the following norm
$$
\|(u, v)\|^{2}_{\xi}=\int_{\R^{N}}  |(-\Delta)^{\frac{s}{2}}u|^{2}+ |(-\Delta)^{\frac{s}{2}}v|^{2}\, dx+\int_{\R^{N}} (V(\xi) u^{2}+W(\xi) v^{2}) dx. 
$$
We introduce the energy functional $\J_{\xi}: \X_{0}\rightarrow \R$ associated to \eqref{P0}, that is 
$$
\J_{\xi}(u, v)=\frac{1}{2}\|(u, v)\|^{2}_{\xi}-\int_{\R^{N}} Q(u, v)+\frac{1}{2^{*}_{s}} K(u, v)\, dx. 
$$
From the growth assumptions on $Q$ and $K$, it is clear that $\J_{\xi}$ is well defined and $\J_{\xi}\in C^{1}(\X_{0}, \R)$.
Let us denote by $m_{\xi}$ the ground state level of $\J_{\xi}$, that is
$$
m_{\xi}=\inf_{(u, v)\in \X_{0}\setminus \{0\}} \max_{t\geq 0} \J_{\xi}(tu, tv)>0.
$$

\noindent
We begin proving the following lemma which will be useful to prove that the weak limit of Palais-Smale sequences of $\J_{\xi}$ are nontrivial.
\begin{lemma}\label{lem4.2}
Let $\{(u_{n}, v_{n})\}\subset \X_{0}$ be a Palais-Smale sequence for $\J_{\xi}$ at the level $d<\frac{s}{N}\widetilde{S}_{K}^{\frac{N}{2s}}$. Then we have either
\begin{compactenum}[(i)]
\item $\|(u_{n}, v_{n})\|_{\xi}\rightarrow 0$, or
\item there exist a sequence $\{y_{n}\}\subset \R^{N}$ and $R, \gamma>0$ such that 
$$
\liminf_{n\rightarrow \infty}\int_{B_{R}(y_{n})} (|u_{n}|^{2}+|v_{n}|^{2}) dx\geq \gamma.
$$
\end{compactenum}
\end{lemma}
\begin{proof}
Suppose that $(ii)$ does not hold. Then, for any $R>0$, we have
$$
\lim_{n\rightarrow \infty} \sup_{y\in \R^{N}}\int_{B_{R}(y)} |u_{n}|^{2} dx=0=\lim_{n\rightarrow \infty} \sup_{y\in \R^{N}}\int_{B_{R}(y)} |v_{n}|^{2} dx.
$$
In view of Lemma \ref{lionslemma}, we can see that 
$$
u_{n}, v_{n}\rightarrow 0 \mbox{ in } L^{r}(\R^{N}) \quad \forall r\in (2, 2^{*}_{s}).
$$
This together with $(Q2)$, imply that 
$$
\int_{\R^{N}} Q(u_{n}, v_{n}) dx \rightarrow 0.
$$
Since $\{(u_{n}, v_{n})\}$ is bounded, we have $\langle \J'_{\xi}(u_{n}, v_{n}),(u_{n}, v_{n})\rangle\rightarrow 0$, so there exists $L\geq 0$ such that
\begin{equation}\label{4.1}
\|(u_{n}, v_{n})\|^{2}_{\xi}\rightarrow L \mbox{ and }  \int_{\R^{N}} K(u_{n}, v_{n}) \,dx\rightarrow L.
\end{equation}
Recalling that $\J_{\xi}(u_{n}, v_{n})\rightarrow d$, we can use \eqref{4.1} to deduce that $d=\frac{Ls}{N}$. 

On the other hand, from the definition of $\widetilde{S}_{K}$, we know that
$$
\|(u_{n}, v_{n})\|^{2}_{\xi}\geq \widetilde{S}_{K} \left(\int_{\R^{N}} K(u_{n}, v_{n})\,dx \right)^{\frac{2}{2^{*}_{s}}}
$$
so, by passing to the limit as $n\rightarrow \infty$ in the above relation, we can deduce that $L\geq \widetilde{S}_{K} L^{\frac{2}{2^{*}_{s}}}$. \\
If $L>0$, we get $Nd=sL\geq s \widetilde{S}_{K}^{\frac{N}{2s}}$, which gives a contradiction. Hence $L=0$ and $(i)$ holds.

\end{proof}

\noindent
Now, we are ready to demonstrate that the above critical autonomous system admits a nontrivial solution.
\begin{theorem}\label{prop4.3}
The problem \eqref{P0} has a weak solution.
\end{theorem}
\begin{proof}
Firstly, we show that 
\begin{equation}\label{FFc}
m_{\xi}<\frac{s}{N} \widetilde{S}_{K}^{\frac{N}{2s}}.
\end{equation}
By using the definition of $m_{\xi}$, it is enough to prove that there exists $(u, v)\in \X_{0}$ such that 
$$
\max_{t\geq 0} \J_{\xi}(tu, tv)<\frac{s}{N} \widetilde{S}_{K}^{\frac{N}{2s}}.
$$
Fix $\eta \in C^{\infty}_{0}(\R^{N})$ a cut-off function such that $0\leq \eta \leq 1$, $\eta=1$ on $B_{r}$ and $\eta=0$ on $\R^{N}\setminus B_{2r}$.

For $\varepsilon>0$, let us define $v_{\varepsilon}(x)=\eta(x)z_{\varepsilon}(x)$, where
$$
z_{\varepsilon}(x)=\frac{\kappa \varepsilon^{\frac{N-2s}{2}}}{(\varepsilon^{2}+|x|^{2})^{\frac{N-2s}{2}}}
$$
is a solution to
$$
(-\Delta)^{s}u=S_{*}|u|^{2^{*}_{s}-2}u \mbox{ in } \R^{N}
$$
and $\kappa$ is a suitable positive constant depending only on $N$ and $s$.\\
Now, we set
$$
u_{\varepsilon}=\frac{z_{\varepsilon}}{\left(\int_{\R^{N}} |z_{\varepsilon}|^{2^{*}_{s}} dx \right)^{\frac{1}{2^{*}_{s}}}}.
$$
Let us recall the following fundamental estimates for $u_{\e}$ (see Proposition $21$ and $22$ in \cite{SV}):
\begin{align}\label{BN1}
\int_{\R^{N}} |(-\Delta)^{\frac{s}{2}} u_{\e}|^{2} dx\leq S_{*}+O(\varepsilon^{N-2s}),
\end{align}
\begin{equation}\label{BN2}
\int_{\R^{N}} |u_{\varepsilon}|^{2} dx=
\left\{
\begin{array}{ll}
O(\varepsilon^{2s})  &\mbox{ if } N>4s \\
O(\varepsilon^{2s} |\log(\varepsilon)|) &\mbox{ if } N=4s \\
O(\varepsilon^{N-2s}) &\mbox{ if } N<4s
\end{array},
\right.
\end{equation}
and
\begin{equation}\label{BN3}
\int_{\R^{N}} |u_{\varepsilon}|^{q} dx=
\left\{
\begin{array}{ll}
O(\varepsilon^{\frac{2N-(N-2s)q}{2}})  &\mbox{ if } q>\frac{N}{N-2s} \\
O(|\log(\varepsilon)|\varepsilon^{\frac{N}{2}}) &\mbox{ if } q=\frac{N}{N-2s} \\
O(\varepsilon^{\frac{(N-2s)q}{2}}) &\mbox{ if } q<\frac{N}{N-2s}.
\end{array}
\right.
\end{equation}
By using Lemma \ref{FSthm}, we know that there exist $A, B\in \R$ such that $\widetilde{S}_{K}$ is attained by $(A u_{\e}, B u_{\e})$ and
\begin{equation}\label{MIO}
\widetilde{S}_{K}=S_{*} \frac{(A^{2}+B^{2})}{(\int_{\R^{N}} K(Az_{\e}, Bz_{\e}) \,dx)^{\frac{2}{2^{*}_{s}}}}.
\end{equation}
From $(Q7)$, we can see that
\begin{align*}
&\J_{\xi}(t A u_{\e}, t  B u_{\e})\\
&\leq \left[\frac{t^{2}}{2} (A^{2}+B^{2}) D_{\e}-\frac{t^{2^{*}_{s}}}{2^{*}_{s}} \int_{\R^{N}} K(A z_{\e}, B z_{\e})\, dx\right]- \lambda t^{q_{1}} A^{q_{1}} B^{q_{1}} \int_{\R^{N}} |u_{\e}|^{q_{1}} dx =: h_{\e}(t), 
\end{align*}
where 
$$
D_{\e}=\int_{\R^{N}} |(-\Delta)^{\frac{s}{2}} u_{\e}|^{2} dx+\int_{\R^{N}} \max\{V(\xi), W(\xi)\} u_{\e}^{2} dx.
$$
Let us denote by $t_{\e}>0$ the maximum point of $h_{\e}(t)$. Since $h'_{\e}(t_{\e})=0$, we have
$$
\bar{t}_{\e}=\left(\frac{D_{\e} (A^2+B^2)}{(\int_{\R^{N}} K(Az_{\e}, Bz_{\e})\, dx)^{\frac{2}{2^{*}_{s}}}}\right)^{\frac{N-2s}{4s}}\geq t_{\e}>0.
$$
Observing that $t\mapsto t^{2} D_{\e} \frac{(A^{2}+B^{2})}{2}-t^{2^{*}_{s}} \int_{\R^{N}} K(Az_{\e}, Bz_{\e}) \,dx$ is increasing in $(0, \bar{t}_{\e})$, we can deduce that
$$
\J_{\xi}(t A u_{\e}, t  B u_{\e})\leq \frac{s}{N}\left(\frac{D_{\e} (A^2+B^2)}{(\int_{\R^{N}}K(Az_{\e}, Bz_{\e}) \,dx)^{\frac{2}{2^{*}_{s}}}}\right)^{\frac{N}{2s}} -\lambda t^{q_{1}} A^{q_{1}} B^{q_{1}} \int_{\R^{N}} |u_{\e}|^{q_{1}} dx.
$$
Now, using \eqref{BN1} and recalling that $(a+b)^{r}\leq a^{r}+r(a+b)^{r-1} b$ for any $a, b>0$ and $r\geq 1$, we can see that
$$
D_{\e}^{N/2s}\leq S_{*}^{N/2s}+O(\e^{N-2s})+ C_{1} \int_{\R^{N}} |u_{\e}|^{2} dx,
$$
On the other hand, the fact that $h'_{\e}(t_{\e})=0$ and the mountain pass geometry of $\J_{\e}$ imply that 
$$
t_{\e}\geq \sigma \mbox{ for any } \e>0, 
$$
for some constant $\sigma>0$ independent of $\e$.\\ 
Then, by using \eqref{MIO}, we have
\begin{align*}
\J_{\xi}(t A u_{\e}, t  B u_{\e})\leq \frac{s}{N} \widetilde{S}_{K}^{\frac{N}{2s}}+O(\e^{N-2s})+ C_{2} \int_{\R^{N}} |u_{\e}|^{2} dx -\lambda C_{3}\int_{\R^{N}} |u_{\e}|^{q_{1}} dx,
\end{align*}
where $C_{2}, C_{3}>0$ are independent of $\e$ and $\lambda$.\\
Now, we consider the following cases: \\
If $N>4s$, then $q_{1}>\frac{N}{N-2s}$. Then, by using (\ref{BN2}) and (\ref{BN3}), we have 
\begin{align*}
\sup_{t\geq 0} h_{\e}(t)&\leq \frac{s}{N} \widetilde{S}_{K}^{\frac{N}{2s}}+O(\varepsilon^{N-2s})+O(\varepsilon^{2s})-\lambda O(\varepsilon^{\frac{2N-(N-2s)q_{1}}{2}}).
\end{align*}
Since $\frac{2N-(N-2s)q_{1}}{2}<2s<N-2s$, we get the conclusion for $\varepsilon$ small enough. \\
When $N=4s$, then $q_{1}\in (2, 4)$ and in particular $q_{1}>\frac{N}{N-2s}=2$, so from (\ref{BN2}) and (\ref{BN3}) we infer 
\begin{align*}
\sup_{t\geq 0} h_{\e}(t)&\leq\frac{s}{N} \widetilde{S}_{K}^{\frac{N}{2s}}+O(\varepsilon^{2s})+O(\varepsilon^{2s}|\log(\varepsilon)|)-\lambda O(\varepsilon^{4s-sq_{1}}).
\end{align*}
Since
$$
\lim_{\varepsilon\rightarrow 0} \frac{\varepsilon^{4s-sq_{1}}}{\varepsilon^{2s}(1+|\log(\varepsilon)|)}=\infty,
$$
we again get the assert for $\varepsilon$ small enough. \\
If $2s<N<4s$ and $q_{1}\in (\frac{4s}{N-2s}, 2^{*}_{s})$, then  $q_{1}>\frac{N}{N-2s}$. Hence 
\begin{align*}
\sup_{t\geq 0} h_{\e}(t)&\leq \frac{s}{N} \widetilde{S}_{K}^{\frac{N}{2s}}+O(\varepsilon^{N-2s})+O(\varepsilon^{N-2s})-\lambda O(\varepsilon^{\frac{2N-(N-2s)q_{1}}{2}})
\end{align*}
and we obtain the conclusion for $\e$ sufficiently small since $\frac{2N-(N-2s)q_{1}}{2}<N-2s$. \\
If $2s<N<4s$ and $q_{1}\in (2, \frac{4s}{N-2s}]$, we argue as before and by using (\ref{BN3}) we get
\begin{equation*}
\sup_{t\geq 0} h_{\e}(t)\leq 
\left\{
\begin{array}{ll}
\frac{s}{N} \widetilde{S}_{K}^{\frac{N}{2s}}+O(\varepsilon^{N-2s})-\lambda O(\varepsilon^{\frac{2N-(N-2s)q_{1}}{2}})  &\mbox{ if } q_{1}>\frac{N}{N-2s} \\
\frac{s}{N} \widetilde{S}_{K}^{\frac{N}{2s}}+O(\varepsilon^{N-2s})-\lambda O(|\log(\varepsilon)|\varepsilon^{\frac{N}{2}}) &\mbox{ if } q_{1}=\frac{N}{N-2s} \\
\frac{s}{N} \widetilde{S}_{K}^{\frac{N}{2s}}+O(\varepsilon^{N-2s})-\lambda O(\varepsilon^{\frac{(N-2s)q_{1}}{2}}) &\mbox{ if } q_{1}<\frac{N}{N-2s}. 
\end{array}
\right.
\end{equation*}
Then, there exists $\lambda_{0}>0$ large enough such that for any $\lambda\geq \lambda_{0}$  and $\e>0$ small it holds
\begin{align*}
\sup_{t\geq 0} h_{\e}(t)< \frac{s}{N} \widetilde{S}_{K}^{\frac{N}{2s}}.
\end{align*}
Taking into account the above estimates, we can infer that for any $\e>0$ sufficiently small 
\begin{align*}
\max_{t\geq 0} \J_{\xi}(t A u_{\e}, t  B u_{\e})\leq \max_{t\geq 0} h_{\e}(t)=h_{\e}(t_{\e})<\frac{s}{N} \widetilde{S}_{K}^{\frac{N}{2s}},
\end{align*}
that is \eqref{FFc} holds.

Now, we observe that $\J_{\xi}$ has a mountain pass geometry. Indeed $\J_{\xi}(0, 0)=0$.
From the assumptions $(Q2)$, $(K3)$ and \eqref{2.1}, and by applying Theorem \ref{Sembedding}, we can see that
$$
\J_{\xi}(u, v)\geq \frac{1}{2}\|(u, v)\|^{2}_{\xi}-C_{1}\|(u, v)\|^{p}_{\xi}-C_{2}\|(u, v)\|^{2^{*}_{s}}_{\xi}
$$ 
so we can find $\alpha, \beta>0$ such that $\J_{\xi}(u, v)\geq \beta$ if $\|(u, v)\|_{\xi}=\alpha$.
Finally, if we take $(\phi_{1}, \phi_{2})\in \X_{0}$ such that $K(\phi_{1}, \phi_{2})>0$, we can use $(Q1)$ and $(K1)$ to see that
\begin{align*}
\J_{\xi}(t u, t v)\leq \frac{t^{2}}{2} \|(\phi_{1}, \phi_{2})\|^{2}_{\xi}-t^{p} \int_{\R^{N}} Q(\phi_{1}, \phi_{2}) dx -\frac{t^{2^{*}_{s}}}{2^{*}_{s}} \int_{\R^{N}} K(\phi_{1}, \phi_{2}) dx\rightarrow -\infty 
\end{align*}
as $t\rightarrow \infty$.
Then, we can use Theorem $1.15$ in \cite{W}, and in view of \eqref{FFc}, we can find  a sequence $\{(u_{n}, v_{n})\}\subset \X_{0}$ such that 
$$
\J_{\xi}(u_{n}, v_{n})\rightarrow m_{\xi}<\frac{s}{N} \widetilde{S}_{K}^{\frac{N}{2s}} \mbox{ and }  \J'_{\xi}(u_{n}, v_{n})\rightarrow 0.
$$
Clearly, $(u_{n}, v_{n})$ is bounded in $\X_{0}$, so we may assume that $(u_{n}, v_{n})\rightharpoonup (u, v)$ in $\X_{0}$ and $u_{n}\rightarrow u$, $v_{n}\rightarrow v$ in $L^{q}_{loc}(\R^{N})$ for any $q\in [2, 2^{*}_{s})$ (by Theorem \ref{Sembedding}). By using $(Q2)$ and $(K3)$, we deduce that
$$
\int_{\R^{N}} (Q_{u}(u_{n}, v_{n}) \phi+Q_{v}(u_{n}, v_{n}) \psi) \, dx\rightarrow   \int_{\R^{N}} (Q_{u}(u, v) \phi+Q_{v}(u, v) \psi) \, dx
$$
and
$$
\int_{\R^{N}} (K_{u}(u_{n}, v_{n}) \phi+K_{v}(u_{n}, v_{n}) \psi) \, dx\rightarrow \int_{\R^{N}}  (K_{u}(u, v) \phi+K_{v}(u, v) \psi) \, dx 
$$
for any $\phi, \psi\in C^{\infty}_{0}(\R^{N})$.
Therefore, we deduce that $\langle \J'_{\xi}(u, v), (\phi, \psi)\rangle=0$ for any $(\phi, \psi)\in \X_{0}$, that is $(u, v)$ is a critical point of $\J_{\xi}$.
Now, we distinguish two cases.
\begin{compactenum}[$(a)$]
\item $u=v=0$;
\item $u\neq 0$, $u\geq 0$ and $v\neq 0$, $v\geq 0$.
\end{compactenum}
Assume that $(a)$ holds. By Lemma \ref{lem4.2}, we have the following two alternatives:
\begin{compactenum}[$(i)$]
\item $\|(u_{n}, v_{n})\|_{\xi}\rightarrow 0$, or
\item there exist a sequence $\{y_{n}\}\subset \R^{N}$ and $R, \gamma>0$ such that 
$$
\liminf_{n\rightarrow \infty}\int_{B_{R}(y_{n})} (|u_{n}|^{2}+|v_{n}|^{2}) dx\geq \gamma.
$$
\end{compactenum}
Since $\J_{\xi}(u_{n}, v_{n})\rightarrow m_{\xi}>0$, it is clear that $(i)$ cannot occur. Therefore,
there exist a sequence $(y_{n})\subset \R^{N}$ and $R, \gamma>0$ such that 
\begin{equation}\label{c}
\liminf_{n\rightarrow \infty}\int_{B_{R}(y_{n})} (|u_{n}|^{2}+|v_{n}|^{2}) dx\geq \gamma.
\end{equation}
Set $(\tilde{u}_{n}(x), \tilde{v}_{n}(x))=(u_{n}(x+y_{n}), v_{n}(x+y_{n}))$.
Since $\{(\tilde{u}_{n}, \tilde{v}_{n})\}$ is bounded, we may assume that $(\tilde{u}_{n}, \tilde{v}_{n})\rightharpoonup (\tilde{u}, \tilde{v})$ in $\X_{0}$. Taking into account $\J_{\xi}(\tilde{u}_{n}, \tilde{v}_{n})=\J_{\xi}(u_{n}, v_{n})$ and $\|\J_{\xi}'(\tilde{u}_{n}, \tilde{v}_{n})\|=o_{n}(1)$, we deduce that $(\tilde{u}_{n}, \tilde{v}_{n})$ is a Palais-Smale of $\J_{\xi}$ at the level $m_{\xi}$ and $(\tilde{u}, \tilde{v})$ is a critical point of $\J_{\xi}$.
From \eqref{c}, we deduce that 
\begin{equation}\label{d}
\liminf_{n\rightarrow \infty}\int_{B_{R}(0)} (|\tilde{u}_{n}|^{2}+|\tilde{v}_{n}|^{2}) dx\geq \gamma.
\end{equation}
Then, by $(Q4)$, we can see that $\tilde{u}$ and $\tilde{v}$ are not zero. 
Indeed, if $u=0$ and $v\neq 0$, then follows by  $\langle \J'_{\xi}(\tilde{u}, \tilde{v}), (\tilde{u}, \tilde{v})\rangle=0$ and \eqref{2.1} that
$$
\|v\|^{2}_{\xi}=\int_{\R^{N}} Q_{v}(0, v)v+\frac{1}{2^{*}_{s}} K_{v}(0, v)v\, dx=\int_{\R^{N}} pQ(0, v)+K(0, v)\, dx=0
$$
which gives a contradiction.\\
Now, we prove that $\J_{\xi}(\tilde{u}, \tilde{v})=m_{\xi}$. In order to achieve our aim, we will show that $(\tilde{u}_{n}, \tilde{v}_{n})\rightarrow (\tilde{u}, \tilde{v})$ in $\X_{0}$ as $n\rightarrow \infty$. Fix $\theta\in (2, p)$.
Let us observe that, from the lower semicontinuity continuity of the $\X_{0}$-norm, we get
\begin{equation}\label{ortega}
\|(\tilde{u}, \tilde{v})\|_{\xi}\leq \liminf_{n\rightarrow \infty} \|(\tilde{u}_{n}, \tilde{v}_{n})\|_{\xi}.
\end{equation}
Then, if we do not have the equality in \eqref{ortega}, by Fatou Lemma and \eqref{2.1}, we can see that 
\begin{align*}
m_{\xi}&\leq \J_{\xi}(\tilde{u}, \tilde{v}) \\
&=\J_{\xi}(\tilde{u}, \tilde{v})-\frac{1}{\theta} \langle \J_{\xi}'(\tilde{u}, \tilde{v}), (\tilde{u}, \tilde{v})\rangle \\
&=\left(\frac{1}{2}-\frac{1}{\theta}\right) \|(\tilde{u}, \tilde{v})\|^{2}_{\xi}+
\int_{\R^{N}} \left(\frac{p}{\theta}-1\right) Q(\tilde{u}, \tilde{v}) \, dx+\int_{\R^{N}} \left(\frac{1}{\theta}-\frac{1}{2^{*}_{s}}\right) K(\tilde{u}, \tilde{v}) \, dx \\
&< \liminf_{n\rightarrow \infty} \Bigl[ \left(\frac{1}{2}-\frac{1}{\theta}\right) \|(\tilde{u}_{n}, \tilde{v}_{n})\|^{2}_{\xi}+
\int_{\R^{N}} \left(\frac{p}{\theta}-1\right) Q(\tilde{u}_{n}, \tilde{v}_{n}) \, dx\\
&\quad +\int_{\R^{N}} \left(\frac{1}{\theta}-\frac{1}{2^{*}_{s}}\right) K(\tilde{u}_{n}, \tilde{v}_{n}) \, dx   \Bigr] \\
&=\liminf_{n\rightarrow \infty} \left[\J_{\xi}(\tilde{u}_{n}, \tilde{v}_{n})-\frac{1}{\theta} \langle \J'(\tilde{u}_{n}, \tilde{v}_{n}), (\tilde{u}_{n}, \tilde{v}_{n})\rangle \right] \\
&=m_{\xi}
\end{align*} 
which gives a contradiction. As a consequence, $\|(\tilde{u}_{n}, \tilde{v}_{n})\|_{\xi} \rightarrow \|(\tilde{u}, \tilde{v})\|_{\xi}$, which implies $(\tilde{u}_{n}, \tilde{v}_{n})\rightarrow (\tilde{u}, \tilde{v})$ in $\X_{0}$. Therefore, $\J_{\xi}(\tilde{u}, \tilde{v})=m_{\xi}$ and this ends the proof.

\end{proof}

\begin{remark}\label{remark}
We note that $\tilde{u}$ and $\tilde{v}$ are continuous and positive in $\R^{N}$.\\
Indeed, by using $\langle \J'_{\xi}(u, v),(u^{-}, v^{-})\rangle=0$, and recalling that $(x-y)(x^{-}-y^{-})\leq -|x^{-}-y^{-}|^{2}$ for any $x, y\in \R$, where $x^{-}=\max\{-x, 0\}$, we can see that
\begin{align*}
0&=\langle \J'_{\xi}(u, v),(u^{-}, v^{-})\rangle\\
&=\iint_{\R^{2N}} \left[\frac{(u(x)-u(y))(u^{-}(x)-u^{-}(y))}{|x-y|^{N+2s}}+ \frac{(v(x)-v(y))(v^{-}(x)-v^{-}(y))}{|x-y|^{N+2s}}\right] dx dy \\
&+\int_{\R^{N}} (V(\xi) u u^{-}+W(\xi) v v^{-}) dx -\int_{\R^{N}} (Q_{u}(u, v)u^{-}+Q_{v}(u, v) v^{-}) dx \\
&-\frac{1}{2^{*}_{s}}\int_{\R^{N}} (K_{u}(u, v)u^{-}+K_{v}(u, v) v^{-}) dx  \\
&\leq -\iint_{\R^{2N}} \left[\frac{|u^{-}(x)-u^{-}(y)|^{2}}{|x-y|^{N+2s}} + \frac{|v^{-}(x)-v^{-}(y)|^{2}}{|x-y|^{N+2s}}\right] dx dy \\
& -\int_{\R^{N}} (V(\xi) (u^{-})^{2}+W(\xi) (v^{-})^{2}) dx = -\|(u^{-}, v^{-})\|^{2}_{\xi},
\end{align*} 
where we used the fact that $Q_{u}=K_{u}=0$ on $(-\infty, 0)\times \R$ and $Q_{v}=K_{v}=0$ on $\R\times (-\infty, 0)$. Therefore $u, v\geq 0$ in $\R^{N}$.
In view of $(Q2)$ and $(K3)$, we can see that $z=u+v\geq 0$ satisfies $(-\Delta)^{s}z+\min\{V(\xi), W(\xi)\} z\leq C_{0} (z^{p-1}+z^{2^{*}_{s}-1})$ in $\R^{N}$, for some $C_{0}>0$ given by $(Q2)$ and $(K3)$,
hence, by using a Moser iteration argument (see for instance Proposition $5.1.1.$ in \cite{DPMV} or Theorem $1.2$ in \cite{A3}) we can prove that $z\in L^{\infty}(\R^{N})$, which implies that $u, v\in L^{\infty}(\R^{N})$. Then $\nabla Q(u,v)$ and $\nabla K(u,v)$ are bounded, and by applying Proposition $2.9$ in \cite{Silvestre} we have $u, v\in C^{0, \alpha}(\R^{N})\cap L^{\infty}(\R^{N})$. From the Harnack inequality \cite{CabSir} we get $u, v>0$ in $\R^{N}$.
\end{remark}

Let us observe that critical points of $\J_{\xi}$ belong to the Nehari manifold
$$
\N_{\xi}=\{(u, v)\in \X_{0}\setminus \{(0, 0)\}:  \langle \J'_{\xi}(u, v),(u, v)\rangle=0\}.
$$
Then, we can see that
$$
m_{\xi}=\inf_{(u, v)\in \N_{\xi}} \J_{\xi}(u, v)=C(\xi).
$$
Moreover, from Theorem \ref{prop4.3}, we deduce that
$$
M=\left\{x\in \R^{N}: C(x)=\inf_{\xi\in \R^{N}} C(\xi)\right\}\neq \emptyset.
$$ 
Arguing as in the proofs of Lemma $2.1$ and Lemma $2.2$ in \cite{A5}, it is easy to show that the following results hold.
\begin{lemma}\label{C0}
The map $\xi\mapsto C(\xi)$ is continuous.
\end{lemma}
\begin{lemma}\label{C00}
$C^{*}=C(x_{0})=\inf_{\xi\in \Lambda} C(\xi)<\min_{\xi\in \partial \Lambda} C(\xi)$.
\end{lemma}

\section{The critical modified problem}
As in \cite{AFF2}, to study solutions of problem \eqref{P} we define a suitable penalization function.
First of all, we can note that by using the change of variable $z\mapsto \e x$ in \eqref{P}, we can consider the following rescaled system
\begin{equation}\label{P'}
\left\{
\begin{array}{ll}
 (-\Delta)^{s}u+V(\e x)u=Q_{u}(u, v)+\frac{1}{2^{*}_{s}} K_{u}(u, v)  &\mbox{ in } \R^{N}\\
 (-\Delta)^{s}u+W(\e x) v=Q_{v}(u, v)+\frac{1}{2^{*}_{s}} K_{v}(u, v)  &\mbox{ in } \R^{N} \\
u, v>0 &\mbox{ in } \R^{N}.
\end{array}
\right.
\end{equation}
Let $a>0$ and we take a non-increasing function $\eta: \R\rightarrow \R$ such that 
\begin{equation}\label{3}
\eta=1 \mbox{ on } (-\infty, a], \, \eta=0 \mbox{ on } [5a, \infty), \, |\eta'(t)|\leq \frac{C}{a}, \, \mbox{ and } |\eta''(t)|\leq \frac{C}{a^{2}}.
\end{equation} 
Then, we introduce the following function $\hat{Q}: \R^{2}\rightarrow \R$ by setting
$$
\hat{Q}(u, v)=\eta(|(u, v)|) \left[Q(u, v)+\frac{1}{2^{*}_{s}} K(u, v)\right] +(1-\eta(|(u, v)|) A(u^{2}+v^{2})
$$
where 
$$
A=\max\left\{ \frac{Q(u, v)+\frac{1}{2^{*}_{s}} K(u, v)}{u^{2}+v^{2}}: (u, v)\in \R^{2}, a\leq |(u, v)|\leq 5a \right\}>0.
$$
Let us observe that $A\rightarrow 0$ as $a\rightarrow 0^{+}$, so we may assume that $A<\frac{\alpha}{4}$, where $\alpha=\min\{V(x_{0}), W(x_{0})\}$.
Now, we define $H: \R^{N}\times \R^{2}\rightarrow \R$ by setting
$$
H(x, u, v)=\chi_{\Lambda}(x) \left[Q(u, v)+\frac{1}{2^{*}_{s}} K(u, v)\right]+(1-\chi_{\Lambda}(x)) \hat{Q}(u, v).
$$
As in \cite{AFF2}, we can prove the following useful properties of the penalized function $H$.
\begin{lemma}\label{lem2.2Alves}
The function $H$ satisfies the following estimates
\begin{equation}\label{2.4A}
pH(x, u, v)\leq u H_{u}(x, u, v)+v H_{v}(x, u, v) \mbox{ for any } x\in \Lambda,
\end{equation}
\begin{equation}\label{2.5A}
2H(x, u, v)\leq u H_{u}(x, u, v)+v H_{v}(x, u, v) \mbox{ for any } x\in \R^{N}\setminus \Lambda.
\end{equation}
Moreover, for any $k>0$ fixed, we can choose $a>0$ sufficiently small, such that 
\begin{equation}\label{2.6A}
u H_{u}(x, u, v)+v H_{v}(x, u, v)\leq \frac{1}{k} (V(x) u^{2}+W(x) v^{2}) \mbox{ for any } x\in \R^{N}\setminus \Lambda.
\end{equation}
\end{lemma}

\noindent
From now on, we will look for weak solutions of the following modified problem
\begin{equation}\label{P''}
\left\{
\begin{array}{ll}
(-\Delta)^{s}u+V(\e x)u=H_{u}(\e x, u, v)  &\mbox{ in } \R^{N}\\
(-\Delta)^{s}u+W(\e x) v=H_{v}(\e x, u, v) &\mbox{ in } \R^{N} \\
u, v>0 &\mbox{ in } \R^{N},
\end{array}
\right.
\end{equation}
which verify
$$
|(u(x), v(x))|\leq a \mbox{ for any } x\in \R^{N}\setminus \Lambda_{\e},
$$
where $\Lambda_{\e}=\{ x\in \R^{N}: \e x\in \Lambda\}$ and $|(u, v)|=\sqrt{u^{2}+v^{2}}$ for any $u, v\in \R$.\\
Indeed, from the definition of $H$ and $\hat{Q}$, one can see that every solution of \eqref{P''} with the above property is  a solution to \eqref{P'}.

\noindent
Now, we introduce the functional setting in which we study our problem.

For any $\e>0$, we define the fractional space
$$
\X_{\e}=\left\{(u, v)\in H^{s}(\R^{N})\times H^{s}(\R^{N}): \int_{\R^{N}} (V(\e x) |u|^{2}+W(\e x) |v|^{2}) dx<\infty\right\}.
$$
endowed with the norm
$$
\|(u, v)\|^{2}_{\e}=\int_{\R^{N}} |(-\Delta)^{\frac{s}{2}} u|^{2}+ |(-\Delta)^{\frac{s}{2}} v|^{2} dx+\int_{\R^{N}} (V(\e x) u^{2}+W(\e x) v^{2}) dx.
$$
In order to get solutions to \eqref{P''}, we seek critical points of the following Euler-Lagrange functional 
$$
\J_{\e}(u, v)=\frac{1}{2}\|(u, v)\|_{\e}^{2}-\int_{\R^{N}} H(\e x, u, v) dx
$$
for any $(u, v)\in \X_{\e}$.\\
Clearly, critical points of $\J_{\e}$ belong to the Nehari manifold 
$$
\N_{\e}=\{(u, v)\in \X_{\e}\setminus \{(0, 0)\}:  \langle \J'_{\e}(u, v),(u, v)\rangle=0\}.
$$
Standard calculations show that for any $(u, v)\in \X_{\e}\setminus (0, 0)$, the function $t\mapsto \J_{\e}(t u, tv)$ achieves its maximum at unique $t_{u}>0$ such that $t_{u}(u, v)\in \N_{\e}$.\\
We observe that $\J_{\e}\in C^{1}(\X_{\e}, \R)$ has a mountain pass geometry, that is
\begin{compactenum}[$(MP1)$]
\item $\J_{\e}(0, 0)=0$;
\item there exist $\alpha, \rho>0$ such that $\J_{\e}(u, v)\geq \alpha$ for $\|(u, v)\|_{\xi}\geq \alpha$; 
\item there exists $e\in \X_{0}\setminus B_{\rho}(0)$ such that $\J_{\e}(e)\leq 0$.
\end{compactenum}
Clearly $(MP1)$ holds.
By using \eqref{2.4A}-\eqref{2.6A}, we have
$$
\int_{\R^{N}} H(\e x, u, v) dx\leq \int_{\Lambda_{\e}} H(\e x, u, v) dx+\frac{1}{k} \int_{\R^{N}} V(\e x) u^{2}+W(\e x) v^{2} dx.
$$
Then, by using $(Q2)$ and Theorem \ref{Sembedding}, we get
\begin{align*}
\J_{\e}(u, v)&\geq \left(\frac{1}{2}-\frac{1}{k} \right) \|(u, v)\|^{2}_{\e}-C\int_{\R^{N}} |u|^{p}+|v|^{p} \,dx-C\int_{\R^{N}} |u|^{2^{*}_{s}}+|v|^{2^{*}_{s}} dx \\
&\geq \left(\frac{1}{2}-\frac{1}{k} \right) \|(u, v)\|^{2}_{\e}-C\|(u, v)\|^{p}_{\e}-C\|(u, v)\|^{2^{*}_{s}}_{\e}
\end{align*}
where $k>1$ is fixed. Hence, $(MP2)$ holds.
In order to verify $(MP3)$, we can note that for any $(\phi_{1}, \phi_{2})\in \X_{\e}$ such that $K(\phi_{1}, \phi_{2})\geq 0$ and $K(\phi_{1}, \phi_{2})\neq 0$, we have
\begin{align*}
&\J_{\e}(t\phi_{1}, t\phi_{2})\\
&\leq \frac{t^{2}}{2}\|(\phi_{1}, \phi_{2})\|^{2}_{\e}-Ct^{p} \int_{\Lambda_{\e}} Q(\phi_{1}, \phi_{2}) dx-\frac{t^{2^{*}_{s}}}{2^{*}_{s}}\int_{\Lambda_{\e}} K(\phi_{1}, \phi_{2})\, dx \rightarrow -\infty \mbox{ as } t\rightarrow \infty.
\end{align*}
Here we used the assumptions $(Q1)$ and $(K1)$.

Since we are interested in getting multiple critical points, we will work with the functional $\J_{\e}$ restricted to the Nehari manifold $\N_{\e}$. 
Our main purpose is to prove that the functional $\J_{\e}$ restricted to $\N_{\e}$ satisfies the Palais-Smale compactness condition at every level $c<\frac{s}{N}\widetilde{S}_{K}^{\frac{N}{2s}}$, where $\widetilde{S}_{K}$ is defined in Section $2$.
Firstly, we prove that such property holds for the unconstrained functional.\\
To do this, we recall the following variant of the Concentration-Compactness Lemma \cite{DPMV, PP}, whose proof is deferred to Appendix.
\begin{definition}
We say that a sequence $\{(u_{n}, v_{n})\}$ is tight in $\dot{H}^{s}(\R^{N})\times \dot{H}^{s}(\R^{N})$ if for every $\delta>0$ there exists $R>0$ such that $\int_{\R^{N}\setminus B_{R}} |(-\Delta)^{\frac{s}{2}}u_{n}|^{2}+ |(-\Delta)^{\frac{s}{2}}v_{n}|^{2} \, dx\leq \delta$ for any $n\in \mathbb{N}$.
\end{definition}
\begin{lemma}\label{CCL}
Let $\{(u_{n}, v_{n})\}$ be a bounded tight sequence in $\dot{H}^{s}(\R^{N})\times \dot{H}^{s}(\R^{N})$ such that $(u_{n}, v_{n})\rightharpoonup (u, v)$ in $\dot{H}^{s}(\R^{N})\times \dot{H}^{s}(\R^{N})$.
Let us assume that
\begin{align}\begin{split}\label{46FS}
&|(-\Delta)^{\frac{s}{2}}u_{n}|^{2}\rightharpoonup \mu \\
&|(-\Delta)^{\frac{s}{2}}v_{n}|^{2}\rightharpoonup \sigma \\
&K(u_{n}, v_{n})\rightharpoonup \nu 
\end{split}\end{align}
in the sense of measure, where $\mu$, $\sigma$ and $\nu$ are bounded non-negative measures on $\R^{N}$. Then, there exists at most a countable set $I$, a family of distinct points $\{x_{i}\}_{i\in I}\subset \R^{N}$ and $\{\mu_{i}\}_{i\in I}, \{\sigma_{i}\}_{i\in I}, \{\nu_{i}\}_{i\in I}\subset (0, \infty)$ such that
\begin{align}
&\nu=K(u, v)+\sum_{i\in I} \nu_{i} \delta_{x_{i}} \label{47FS}\\
&\mu\geq |(-\Delta)^{\frac{s}{2}}u|^{2}+\sum_{i\in I} \mu_{i} \delta_{x_{i}} \label{48FS}\\
&\sigma\geq |(-\Delta)^{\frac{s}{2}}v|^{2}+\sum_{i\in I} \sigma_{i} \delta_{x_{i}}. \label{49FS}
\end{align}
Moreover, the following inequality holds true 
\begin{align}\label{50FS}
\mu_{i}+\sigma_{i}\geq \widetilde{S}_{K} \nu_{i}^{\frac{2}{2^{*}_{s}}}.
\end{align}
\end{lemma}

\noindent
Now, we show that the following result holds.
\begin{lemma}\label{lemma3.2-1}
Let $\{(u_{n}, v_{n})\}$ be a sequence in $\X_{\e}$ such that 
$$
\J_{\e}(u_{n}, v_{n})\rightarrow c<\frac{s}{N} \widetilde{S}_{K}^{\frac{N}{2s}} \mbox{ and } \J_{\e}'(u_{n}, v_{n})\rightarrow 0.
$$
Then, $\{(u_{n}, v_{n})\}$ admits a convergent subsequence.
\end{lemma}

\begin{proof} 
We begin proving that $\{(u_{n}, v_{n})\}$ is bounded in $\X_{\e}$. 
From the conditions \eqref{2.4A} and \eqref{2.5A}, it follows that
\begin{align}\label{3.1-1}
\int_{\R^{N}} &|(-\Delta)^{\frac{s}{2}} u_{n}|^{2}+ |(-\Delta)^{\frac{s}{2}} v_{n}|^{2} + V(\e x)u_{n}^{2} + W(\e x)v_{n}^{2} \, dx \nonumber \\
&\geq \int_{\Lambda_{\e}} u_{n} H_{u}(\e x, u_{n}, v_{n})+ v_{n} H_{v}(\e x, u_{n}, v_{n})\, dx + o(\| (u_{n}, v_{n})\|_{\e}) \nonumber\\
&\geq \int_{\Lambda_{\e}} pH(\e x, u_{n}, v_{n})\, dx+o(\| (u_{n}, v_{n})\|_{\e}). 
\end{align} 
On the other hand
\begin{align*}
\frac{1}{2} \int_{\R^{N}} &|(-\Delta)^{\frac{s}{2}} u_{n}|^{2}+ |(-\Delta)^{\frac{s}{2}} v_{n}|^{2} + V(\e x)u_{n}^{2} + W(\e x)v_{n}^{2} \, dx \\
& =\int_{\R^{N}} H(\e x, u_{n}, v_{n}) \, dx + O(1), 
\end{align*}
so, by \eqref{2.5A} and \eqref{2.6A}, we get 
\begin{align}\begin{split}\label{3.2-1}
\frac{1}{2} \int_{\R^{N}} &|(-\Delta)^{\frac{s}{2}} u_{n}|^{2}+ |(-\Delta)^{\frac{s}{2}} v_{n}|^{2} + V(\e x)u_{n}^{2} + W(\e x)v_{n}^{2} \, dx\\
&\leq \int_{\Lambda_{\e}} H(\e x, u_{n}, v_{n}) \, dx + \frac{1}{2k} \int_{\R^{N}\setminus \Lambda_{\e}} [V(\e x) u_{n}^{2} + W(\e x) v_{n}^{2}] + O(1). 
\end{split}\end{align}
Putting together \eqref{3.1-1} and \eqref{3.2-1} we obtain
\begin{align*}
\left( \frac{1}{2}- \frac{1}{p}\right) &\int_{\R^{N}} |(-\Delta)^{\frac{s}{2}} u_{n}|^{2}+ |(-\Delta)^{\frac{s}{2}} v_{n}|^{2} + V(\e x)u_{n}^{2} + W(\e x)v_{n}^{2} \, dx \\
&\leq \frac{1}{2k} \int_{\R^{N}\setminus \Lambda_{\e}} [V(\e x) u_{n}^{2} + W(\e x) v_{n}^{2}] + o(\|(u_{n}, v_{n})\|_{\e})+ O(1).
\end{align*}
Taking $k>0$ such that $k>[2(\frac{1}{2}- \frac{1}{p})]^{-1}$, we can deduce that $\{(u_{n}, v_{n})\}$ is bounded. \\
Since $\X_{\e}$ is reflexive, there exists $(u, v)\in \X_{\e}$ and a subsequence, still denoted by $\{(u_{n}, v_{n})\}$, such that $\{(u_{n}, v_{n})\}$ is weakly convergent to $(u, v)$ and strongly in $L^{q}_{loc}(\R^{N})$ for any $q\in [2, 2^{*}_{s})$. \\
Thus, by using $\langle \J'_{\e}(u_{n}, v_{n}), (u_{n}, v_{n})\rangle=o_{n}(1)$, we can see that
\begin{align}\label{ionew}
\|(u_{n}, v_{n})\|^{2}_{\e}=\int_{\R^{N}} u_{n} H_{u}(\e x, u_{n},v_{n})+ v_{n}H_{v}(\e x, u_{n}, v_{n})\, dx+o_{n}(1).
\end{align}
On the other hand, standard calculations show that $(u, v)$ is a critical point of $\J_{\e}$ and it holds
\begin{align}\label{io}
\|(u, v)\|^{2}_{\e}=\int_{\R^{N}} u H_{u}(\e x, u, v)+ v H_{v}(\e x, u, v)\, dx.
\end{align}
Now, we aim to show that $\{(u_{n}, v_{n})\}$ converges strongly to $(u, v)$ in $\X_{\e}$. \\
In order to achieve our purpose, it is enough to show that $\|(u_{n}, v_{n})\|_{\e}\rightarrow \|(u, v)\|_{\e}$, that in view of \eqref{ionew} and \eqref{io}, means to prove that
\begin{equation}\label{ciuccio}
\int_{\R^{N}} u_{n} H_{u}(\e x, u_{n},v_{n})+ v_{n}H_{v}(\e x, u_{n}, v_{n})\, dx\rightarrow \int_{\R^{N}} u H_{u}(\e x, u, v)+ v H_{v}(\e x, u, v)\, dx.
\end{equation}
We begin proving that for each $\delta>0$, there exists $R>0$ such that
\begin{equation}\label{3.3-1}
\limsup_{n\rightarrow \infty} \int_{\R^{N} \setminus B_{R}} |(-\Delta)^{\frac{s}{2}} u_{n}|^{2}+ |(-\Delta)^{\frac{s}{2}} v_{n}|^{2} + V(\e x)u_{n}^{2} + W(\e x)v_{n}^{2} \, dx \leq \delta. 
\end{equation} 
Let $\eta_{R}$ be a cut-off function such that $\eta_{R}=0$ on $B_{R}$, $\eta_{R}=1$ on $\R^{N} \setminus B_{2R}$, $0\leq \eta\leq 1$ and $|\nabla \eta_{R}|\leq \frac{c}{R}$. 
Suppose that $R$ is chosen so that $\Lambda_{\e} \subset B_{R}$.
Since $\{(u_{n}, v_{n})\}$ is a bounded (PS) sequence, we have
\begin{equation*}
\langle \J_{\e}'(u_{n}, v_{n}), (\eta_{R} u_{n}, \eta_{R} v_{n})  \rangle = o_{n}(1). 
\end{equation*}
Hence, by using \eqref{2.6A} with $k>1$, we get
\begin{align*}
&\iint_{\R^{2N}} \eta_{R}(x) \left[\frac{|u_{n}(x)-u_{n}(y)|^{2}}{|x-y|^{N+2s}}+\frac{|v_{n}(x)-v_{n}(y)|^{2}}{|x-y|^{N+2s}}\right]  dx dy\\
&+\iint_{\R^{2N}} \frac{(\eta_{R}(x)-\eta_{R}(y))(u_{n}(x)-u_{n}(y))}{|x-y|^{N+2s}} u_{n}(y) dx dy \\
&+ \int_{\R^{N}} (V(\e x)u_{n}^{2} + W(\e x) v_{n}^{2})\eta_{R} \, dx \\
&=\int_{\R^{N}} (u_{n}H_{u}(\e x, u_{n}, v_{n}) +v_{n}H_{v}(\e x, u_{n}, v_{n}))\eta_{R} + o_{n}(1)\\
&\leq \frac{1}{k} \int_{\R^{N}\setminus \Lambda_{\e}} V(\e x)u_{n}^{2} + W(\e x) v_{n}^{2} \, dx + o_{n}(1), 
\end{align*}
which implies that
\begin{align}\label{jt}
&\left(1-\frac{1}{k}\right)\int_{\R^{N}\setminus B_{R}} |(-\Delta)^{\frac{s}{2}}u_{n}|^{2} + |(-\Delta)^{\frac{s}{2}} v_{n}|^{2}+ V(\e x)u_{n}^{2} + W(\e x) v_{n}^{2} \, dx \nonumber \\
&\leq - \iint_{\R^{2N}} \frac{(\eta_{R}(x)-\eta_{R}(y))(u_{n}(x)-u_{n}(y))}{|x-y|^{N+2s}} u_{n}(y) dx dy+ o_{n}(1).
\end{align}
By using H\"older inequality and the boundedness of $\{(u_{n}, v_{n})\}$, we can see that
\begin{align*}
&\left| \iint_{\R^{2N}} \frac{(\eta_{R}(x)-\eta_{R}(y))(u_{n}(x)-u_{n}(y))}{|x-y|^{N+2s}} u_{n}(y) dx dy\right| \\
&\leq \left( \iint_{\R^{2N}}\frac{|\eta_{R}(x)-\eta_{R}(y)|^{2}}{|x-y|^{N+2s}} u^{2}_{n}(y) dx dy\right)^{\frac{1}{2}} \left( \iint_{\R^{2N}}\frac{|u_{n}(x)-u_{n}(y)|^{2}}{|x-y|^{N+2s}} dx dy\right)^{\frac{1}{2}} \\
&\leq C  \left( \iint_{\R^{2N}}\frac{|\eta_{R}(x)-\eta_{R}(y)|^{2}}{|x-y|^{N+2s}} u^{2}_{n}(y) dx dy\right)^{\frac{1}{2}}.
\end{align*} 
Then, we can argue as in \cite{A5} (see formula $3.13$ there) or Lemma $3.4$ in \cite{AI}, to deduce that
\begin{align}\label{JT}
\lim_{R\rightarrow \infty} \limsup_{n\rightarrow \infty} \iint_{\R^{2N}}\frac{|\eta_{R}(x)-\eta_{R}(y)|^{2}}{|x-y|^{N+2s}} u^{2}_{n}(x) dx dy=0.
\end{align}
Putting together \eqref{jt} and \eqref{JT}, we can deduce that  \eqref{3.3-1} holds.\\
Now, by using \eqref{3.3-1} and \eqref{2.6A} of Lemma \ref{lem2.2Alves}, we can see that 
\begin{equation}\label{3.6-1}
\int_{\R^{N} \setminus B_{R}} u_{n} H_{u}(\e x, u_{n},v_{n})+ v_{n}H_{v}(\e x, u_{n}, v_{n})\, dx \leq \frac{\delta}{4}, 
\end{equation}
for any $n$ big enough.
On the other hand, choosing $R$ large enough, we may assume that
\begin{equation}\label{3.6-2}
\int_{\R^{N} \setminus B_{R}} u H_{u}(\e x, u, v)+ v H_{v}(\e x, u, v)\, dx \leq \frac{\delta}{4}. 
\end{equation}
From the arbitrariness of $\delta>0$, we can see that \eqref{3.6-1} and \eqref{3.6-2} yield
\begin{align}\label{ioo} 
\int_{\R^{N} \setminus B_{R}} &u_{n} H_{u}(\e x, u_{n},v_{n})+ v_{n}H_{v}(\e x, u_{n}, v_{n})\, dx \nonumber\\
&\rightarrow \int_{\R^{N} \setminus B_{R}} u H_{u}(\e x, u, v)+ vH_{v}(\e x, u, v)\, dx
\end{align}
as $n\rightarrow \infty$.
Since $B_{R}\cap (\R^{N}\setminus \Lambda_{\e})$ is bounded, we can use \eqref{2.6A} of Lemma \ref{lem2.2Alves}, the Dominated Convergence Theorem and the strong convergence in $L^{q}_{loc}(\R^{N})$ to see that 
\begin{align}\label{iooo}
\int_{B_{R}\cap (\R^{N}\setminus \Lambda_{\e})} &u_{n} H_{u}(\e x, u_{n},v_{n})+ v_{n}H_{v}(\e x, u_{n}, v_{n})\, dx \nonumber\\
&\rightarrow \int_{B_{R}\cap (\R^{N}\setminus \Lambda_{\e})} u H_{u}(\e x, u, v)+ v H_{v}(\e x, u, v)\, dx
\end{align}
as $n\rightarrow \infty$.\\
At this point, we show that
\begin{equation}\label{ioooo}
\lim_{n\rightarrow \infty}\int_{\Lambda_{\e}} K(u_{n}, v_{n}) \,dx=\int_{\Lambda_{\e}} K(u, v) \,dx.
\end{equation}
Indeed, if we assume that \eqref{ioooo} is true, from Theorem \ref{Sembedding}, $(Q2)$ and $(K3)$, and the Dominated Convergence Theorem, we can see that
\begin{equation}\label{i5}
\int_{B_{R}\cap \Lambda_{\e}} u_{n} H_{u}(\e x, u_{n},v_{n})+ v_{n} H_{v}(\e x, u_{n}, v_{n})\, dx\rightarrow \int_{B_{R}\cap \Lambda_{\e}} u H_{u}(\e x, u, v)+ v H_{v}(\e x, u, v)\, dx.
\end{equation}
Putting together \eqref{ioo},\eqref{iooo} and \eqref{i5}, we can conclude that \eqref{ciuccio} holds. \\
It remains to prove that \eqref{ioooo} is satisfied. 
By Lemma \ref{CCL}, we can find an at most countable index set $I$, sequences $\{x_{i}\}\subset \R^{N}$, $\{\mu_{i}\}, \{\sigma_{i}\}, \{\nu_{i}\}\subset (0, \infty)$ such that 
\begin{align}\label{CML}
&\mu\geq |(-\Delta)^{\frac{s}{2}}u|^{2}+\sum_{i\in I} \mu_{i} \delta_{x_{i}}, \quad \sigma\geq |(-\Delta)^{\frac{s}{2}}v|^{2}+\sum_{i\in I} \sigma_{i} \delta_{x_{i}} \nonumber \\
&\nu=K(u, v)+\sum_{i\in I} \nu_{i} \delta_{x_{i}} \quad \mbox{ and } \widetilde{S}_{K} \nu_{i}^{\frac{2}{2^{*}_{s}}}\leq \mu_{i}+\sigma_{i}
\end{align}
for any $i\in I$, where $\delta_{x_{i}}$ is the Dirac mass at the point $x_{i}$.
Let us show that $\{x_{i}\}_{i\in I}\cap \Lambda_{\e}=\emptyset$. Assume by contradiction that 
$x_{i}\in \Lambda_{\e}$ for some $i\in I$. For any $\rho>0$, we define $\psi_{\rho}(x)=\psi(\frac{x-x_{i}}{\rho})$ where $\psi\in C^{\infty}_{0}(\R^{N}, [0, 1])$ is such that $\psi=1$ in $B_{1}$, $\psi=0$ in $\R^{N}\setminus B_{2}$ and $\|\nabla \psi\|_{L^{\infty}(\R^{N})}\leq 2$. We suppose that  $\rho>0$ is such that $supp(\psi_{\rho})\subset \Lambda_{\e}$. Since $\{(\psi_{\rho} u_{n}, \psi_{\rho} v_{n})\}$ is bounded, we can see that $\langle \J'_{\e}(u_{n}, v_{n}),(\psi_{\rho} u_{n}, \psi_{\rho} v_{n})\rangle=o_{n}(1)$, so, by using \eqref{2.1}, we get
\begin{align}\label{amici}
\iint_{\R^{2N}} &\psi_{\rho}(y)\frac{|u_{n}(x)-u_{n}(y)|^{2}}{|x-y|^{N+2s}}+\psi_{\rho}(y)\frac{|v_{n}(x)-v_{n}(y)|^{2}}{|x-y|^{N+2s}} \, dx dy \nonumber\\
&\leq -\iint_{\R^{2N}} u_{n}(x) \frac{(u_{n}(x)-u_{n}(y)(\psi_{\rho}(x)-\psi_{\rho}(y)))}{|x-y|^{N+2s}} \nonumber\\
&\quad +v_{n}(x) \frac{(v_{n}(x)-v_{n}(y)(\psi_{\rho}(x)-\psi_{\rho}(y)))}{|x-y|^{N+2s}} \, dx dy \nonumber\\
&\quad +\int_{\R^{N}} u_{n}\psi_{\rho} Q_{u}(u_{n}, v_{n})+v_{n}\psi_{\rho} Q_{v}(u_{n}, v_{n})\, dx\nonumber \\
&\quad+\int_{\R^{N}} \psi_{\rho} K(u_{n}, v_{n})\, dx+o_{n}(1).
\end{align}
Due to the fact that $Q$ has subcritical growth and $\psi_{\rho}$ has compact support, we can see that
\begin{align}\label{Alessia1}
\lim_{\rho\rightarrow 0} &\lim_{n\rightarrow \infty} \int_{\R^{N}} u_{n}\psi_{\rho} Q_{u}(u_{n}, v_{n})+v_{n}\psi_{\rho} Q_{u}(u_{n}, v_{n})\, dx\nonumber \\
&=\lim_{\rho\rightarrow 0} \int_{\R^{N}} u\psi_{\rho} Q_{u}(u, v)+v\psi_{\rho} Q_{v}(u, v)\, dx=0.
\end{align}
Now, we show that
\begin{equation}\label{nio}
\lim_{\rho\rightarrow 0}\lim_{n\rightarrow \infty} \iint_{\R^{2N}} u_{n}(x) \frac{(u_{n}(x)-u_{n}(y)(\psi_{\rho}(x)-\psi_{\rho}(y)))}{|x-y|^{N+2s}}\, dx dy=0.
\end{equation}
By using H\"older inequality and the fact that $\{(u_{n}, v_{n})\}$ is bounded in $\X_{\e}$, we can see that
\begin{align*}
&\left| \iint_{\R^{2N}} u_{n}(x) \frac{(u_{n}(x)-u_{n}(y)(\psi_{\rho}(x)-\psi_{\rho}(y)))}{|x-y|^{N+2s}} \, dx dy \right|  \\
&\leq \left(\iint_{\R^{2N}} \frac{|u_{n}(x)-u_{n}(y)|^{2}}{|x-y|^{N+2s}} \, dx dy  \right)^{\frac{1}{2}} \left(\iint_{\R^{2N}} u^{2}_{n}(x) \frac{|\psi_{\rho}(x)-\psi_{\rho}(y)|^{2}}{|x-y|^{N+2s}} \, dx dy \right)^{\frac{1}{2}}\\
&\leq C\left(\iint_{\R^{2N}} u^{2}_{n}(x) \frac{|\psi_{\rho}(x)-\psi_{\rho}(y)|^{2}}{|x-y|^{N+2s}} \, dx dy \right)^{\frac{1}{2}}.
\end{align*}
Therefore, if we prove that
\begin{equation}\label{NIOO}
\lim_{\rho\rightarrow 0}\lim_{n\rightarrow \infty}  \iint_{\R^{2N}} u^{2}_{n}(x) \frac{|\psi_{\rho}(x)-\psi_{\rho}(y)|^{2}}{|x-y|^{N+2s}} \, dx dy =0,
\end{equation}
then  \eqref{nio} is satisfied. \\
Let us note that $\R^{2N}$ can be written as 
\begin{align*}
\R^{2N}&=((\R^{N}-B_{2\rho}(x_{i}))\times (\R^{N}-B_{2\rho}(x_{i})))\cup (B_{2\rho}(x_{i})\times \R^{N}) \cup ((\R^{N}-B_{2\rho}(x_{i}))\times B_{2\rho}(x_{i}))\\
&=: X^{1}_{\rho}\cup X^{2}_{\rho} \cup X^{3}_{\rho}.
\end{align*}
Then
\begin{align}\label{Pa1}
&\iint_{\R^{2N}} u_{n}^{2}(x) \frac{(\psi_{\rho}(x)-\psi_{\rho}(y))^{2}}{|x-y|^{N+2s}} \, dx dy \nonumber\\
&=\iint_{X^{1}_{\rho}} u_{n}^{2}(x)\frac{|\psi_{\rho}(x)-\psi_{\rho}(y)|^{2}}{|x-y|^{N+2s}} \, dx dy +\iint_{X^{2}_{\rho}} u_{n}^{2}(x)\frac{|\psi_{\rho}(x)-\psi_{\rho}(y)|^{2}}{|x-y|^{N+2s}} \, dx dy \nonumber\\
&+ \iint_{X^{3}_{\rho}} u_{n}^{2}(x)\frac{|\psi_{\rho}(x)-\psi_{\rho}(y)|^{2}}{|x-y|^{N+2s}} \, dx dy.
\end{align}
In what follows, we estimate each integral in (\ref{Pa1}).
Since $\psi=0$ in $\R^{N}\setminus B_{2}$, we have
\begin{align}\label{Pa2}
\iint_{X^{1}_{\rho}} u_{n}^{2}(x)\frac{|\psi_{\rho}(x)-\psi_{\rho}(y)|^{2}}{|x-y|^{N+2s}} \, dx dy=0.
\end{align}
Since $0\leq \psi\leq 1$, we can see that
\begin{align}\label{Pa3}
&\iint_{X^{2}_{\rho}} u_{n}^{2}(x)\frac{|\psi_{\rho}(x)-\psi_{\rho}(y)|^{2}}{|x-y|^{N+2s}} \, dx dy \nonumber\\
&=\int_{B_{2\rho}(x_{i})} \,dx \int_{\{y\in \R^{N}: |x-y|\leq \rho\}} u_{n}^{2}(x)\frac{|\psi_{\rho}(x)-\psi_{\rho}(y)|^{2}}{|x-y|^{N+2s}} \, dy \nonumber \\
&+\int_{B_{2\rho}(x_{i})} \, dx \int_{\{y\in \R^{N}: |x-y|> \rho\}} u_{n}^{2}(x)\frac{|\psi_{\rho}(x)-\psi_{\rho}(y)|^{2}}{|x-y|^{N+2s}} \, dy  \nonumber\\
&\leq \rho^{-2} \|\nabla \psi\|_{L^{\infty}(\R^{N})}^{2} \int_{B_{2\rho}(x_{i})} \, dx \int_{\{y\in \R^{N}: |x-y|\leq \rho\}} \frac{u^{2}_{n}(x)}{|x-y|^{N+2s-2}} \, dy\nonumber \\
&+ 4 \int_{B_{2\rho}(x_{i})} \, dx \int_{\{y\in \R^{N}: |x-y|> \rho\}} \frac{u_{n}^{2}(x)}{|x-y|^{N+2s}} \, dy \nonumber\\
&\leq c_{1} \rho^{-2s} \int_{B_{2\rho}(x_{i})} u_{n}^{2}(x) \, dx+c_{2} \rho^{-2s} \int_{B_{2\rho}(x_{i})} u_{n}^{2}(x) \, dx \nonumber \\
&=c_{3} \rho^{-2s} \int_{B_{2\rho}(x_{i})} u_{n}^{2}(x) \, dx,
\end{align}
for some $c_{1}, c_{2}, c_{3}>0$ independent of $\rho$ and $n$.
On the other hand
\begin{align}\label{Pa4}
&\iint_{X^{3}_{\rho}} u_{n}^{2}(x)\frac{|\psi_{\rho}(x)-\psi_{\rho}(y)|^{2}}{|x-y|^{N+2s}} \, dx dy \nonumber\\
&=\int_{\R^{N}\setminus B_{2\rho}(x_{i})} \, dx \int_{\{y\in B_{2\rho}(x_{i}): |x-y|\leq \rho\}} u_{n}^{2}(x)\frac{|\psi_{\rho}(x)-\psi_{\rho}(y)|^{2}}{|x-y|^{N+2s}} \, dy \nonumber\\
&+\int_{\R^{N}\setminus B_{2\rho}(x_{i})} \,dx \int_{\{y\in B_{2\rho}(x_{i}): |x-y|> \rho\}} u_{n}^{2}(x)\frac{|\psi_{\rho}(x)-\psi_{\rho}(y)|^{2}}{|x-y|^{N+2s}} \, dy=: A_{\rho, n}+ B_{\rho, n}. 
\end{align}
Now, we note that $|x-y|<\rho$ and $|y-x_{i}|<2\rho$ imply $|x-x_{i}|<3\rho$, so we get
\begin{align}\label{Pa5}
A_{\rho, n}&\leq \rho^{-2} \|\nabla \psi\|_{L^{\infty}(\R^{N})}^{2} \int_{B_{3\rho}(x_{i})} \, dx \int_{\{y\in B_{2\rho}(x_{i}): |x-y|\leq \rho\}} \frac{u_{n}^{2}(x)}{|x-y|^{N+2s-2}} \, dy \nonumber\\
&\leq \rho^{-2} \|\nabla \psi\|_{L^{\infty}(\R^{N})}^{2} \omega_{N-1} \int_{B_{3\rho}(x_{i})} u_{n}^{2}(x) \, dx \int_{0}^{\rho} \frac{1}{r^{2s-1}} \, dr \nonumber\\
&=c_{4} \rho^{-2s} \int_{B_{3\rho}(x_{i})} u_{n}^{2}(x) \, dx,
\end{align}
for some $c_{4}>0$ independent of $\rho$ and $n$. 
Here $\omega_{N-1}$ is the Lebesgue measure of the unit sphere in $\R^{N}$.
Let us observe, that for all $K>4$ it holds 
$$
(\R^{N}\setminus B_{2\rho}(x_{i}))\times B_{2\rho}(x_{i}) \subset (B_{K\rho}(x_{i})\times B_{2\rho}(x_{i})) \cup ((\R^{N}\setminus B_{K\rho}(x_{i}))\times B_{2\rho}(x_{i})).
$$
Then, we have the following estimates
\begin{align}\label{Pa6}
&\int_{B_{K\rho}(x_{i})} \, dx \int_{\{y\in B_{2\rho}(x_{i}): |x-y|> \rho\}} u_{n}^{2}(x)\frac{|\psi_{\rho}(x)-\psi_{\rho}(y)|^{2}}{|x-y|^{N+2s}} \, dy \nonumber\\
&\leq 4\int_{B_{K\rho}(x_{i})} \, dx \int_{\{y\in B_{2\rho}(x_{i}): |x-y|> \rho\}} u_{n}^{2}(x)\frac{1}{|x-y|^{N+2s}} \, dy \nonumber \\
&\leq 4 \int_{B_{K\rho}(x_{i})} \, dx \int_{\{z\in \R^{N}: |z|> \rho\}} u_{n}^{2}(x)\frac{1}{|z|^{N+2s}} \, dz \nonumber\\
&= c_{5} \rho^{-2s} \int_{B_{K\rho}(x_{i})} u_{n}^{2}(x) \, dx,
\end{align}
for some $c_{5}$ independent of $\rho$ and $n$.
On the other hand, $|x-x_{i}|\geq K\rho$ and $|y-x_{i}|<2\rho$ imply 
$$
|x-y|\geq |x-x_{i}|-|y-x_{i}|\geq \frac{|x-x_{i}|}{2}+\frac{K\rho}{2}-2\rho>\frac{|x-x_{i}|}{2},
$$
so we can see that
\begin{align}\label{Pa7}
&\int_{\R^{N}\setminus B_{K\rho}(x_{i})} \, dx \int_{\{y\in B_{2\rho}(x_{i}): |x-y|>\rho\}} u_{n}^{2}(x)\frac{|\psi_{\rho}(x)-\psi_{\rho}(y)|^{2}}{|x-y|^{N+2s}} \, dy \nonumber\\
&\leq  2^{N+2s} 4 \int_{\R^{N}\setminus B_{K\rho}(x_{i})} \, dx \int_{\{y\in B_{2\rho}(x_{i}): |x-y|>\rho\}} \frac{u_{n}^{2}(x)}{|x-x_{i}|^{N+2s}} \, dy \nonumber\\
&\leq 2^{N+2s} 4 \left(2\rho\right)^{N} \int_{\R^{N}\setminus B_{K\rho}(x_{i})} \frac{u_{n}^{2}(x)}{|x-x_{i}|^{N+2s}} \, dx \nonumber\\
&\leq 2^{2(N+s+1)}\! \rho^{N}\!\left(\int_{\R^{N}\setminus B_{K\rho}(x_{i})} \!\!\!u_{n}^{2^{*}_{s}}(x) \, dx\right)^{\frac{2}{2^{*}_{s}}}\!\! \left(\int_{\R^{N}\setminus B_{K\rho}(x_{i})} \!\!\!|x-x_{i}|^{-(N+2s)\frac{2^{*}_{s}}{2^{*}_{s}-2}} \, dx\right)^{\frac{2^{*}_{s}-2}{2^{*}_{s}}} \nonumber\\
&\leq c_{6} K^{-N} \left(\int_{\R^{N}\setminus B_{K\rho}(x_{i})} u_{n}^{2^{*}_{s}}(x) \, dx\right)^{\frac{2}{2^{*}_{s}}},
\end{align}
for some $c_{6}>0$ independent of $\rho$ and $n$.
Taking into account (\ref{Pa6}) and (\ref{Pa7}), and  the fact that $\{u_{n}\}$ is bounded in $L^{2^{*}_{s}}(\R^{N})$, we can find $c_{7}>0$  independent of $\rho$ and $n$ such that 
\begin{align}\label{Pa8}
B_{\rho, n}\leq c_{5} \rho^{-2s} \int_{B_{K\rho}(x_{i})} u_{n}^{2}(x) \, dx+c_{7} K^{-N}.
\end{align}
Putting together (\ref{Pa1})-(\ref{Pa5}) and (\ref{Pa8}), we have
\begin{align}\label{stima}
\iint_{\R^{2N}} u_{n}^{2}(x)\frac{|\psi_{\rho}(x)-\psi_{\rho}(y)|^{2}}{|x-y|^{N+2s}} \, dx dy \leq c_{8} \rho^{-2s} \int_{B_{K\rho}(x_{i})} u_{n}^{2}(x) \, dx+c_{9} K^{-N},
\end{align}
for some $c_{8}, c_{9}>0$ independent of $\rho$ and $n$.
Since $u_{n}\rightarrow u$ strongly in $L^{2}_{loc}(\R^{N})$, we can deduce that
\begin{align*}
\lim_{n\rightarrow \infty}  c_{8} \rho^{-2s} \int_{B_{K\rho}(x_{i})} u_{n}^{2}(x) \, dx+c_{9} K^{-N} =c_{8} \rho^{-2s} \int_{B_{K\rho}(x_{i})} u^{2}(x) \, dx+c_{9}K^{-N}.
\end{align*}
Moreover, by using H\"older inequality, we get
\begin{align*}
c_{8} \rho^{-2s} & \int_{B_{K\rho}(x_{i})} u^{2}(x) \, dx+c_{9}K^{-N} \\
&\leq c_{8} \rho^{-2s} \left(\int_{B_{K\rho}(x_{i})} u^{2}(x) \, dx\right)^{\frac{2}{2^{*}_{s}}} |B_{K\rho}(x_{i})|^{1-\frac{2}{2^{*}_{s}}}+c_{9} K^{-N} \\
&\leq c_{10} K^{2s}  \left(\int_{B_{K\rho}(x_{i})} u^{2}(x) \, dx\right)^{\frac{2}{2^{*}_{s}}}+c_{9}K^{-N}\rightarrow c_{9}K^{-N} \mbox{ as } \rho\rightarrow 0.
\end{align*}
As a consequence 
\begin{align*}
&\lim_{\rho\rightarrow 0} \limsup_{n\rightarrow \infty} \iint_{\R^{2N}} u_{n}^{2}(x)\frac{|\psi_{\rho}(x)-\psi_{\rho}(y)|^{2}}{|x-y|^{N+2s}} \, dx dy
 \nonumber\\
&=\lim_{K\rightarrow \infty}\lim_{\rho\rightarrow 0} \limsup_{n\rightarrow \infty} \iint_{\R^{2N}} u_{n}^{2}(x)\frac{|\psi_{\rho}(x)-\psi_{\rho}(y)|^{2}}{|x-y|^{N+2s}} \, dx dy =0,
\end{align*}
that is \eqref{NIOO} holds.\\
Now, by using \eqref{CML} and taking the limit as $n\rightarrow \infty$ and $\rho\rightarrow 0$ in \eqref{amici}, we can deduce that \eqref{Alessia1} and \eqref{nio} yield $\nu_{i}\geq \mu_{i}+\sigma_{i}$.
In view of the last statement in \eqref{CML}, we have $\nu_{i}\geq \widetilde{S}_{K}^{\frac{2}{2^{*}_{s}}}$, and by using Lemma \ref{lem2.2Alves}, \eqref{2.1} and $p>2$, we get
\begin{align*}
c&=\J_{\e}(u_{n}, v_{n})-\frac{1}{2}\langle \J'_{\e}(u_{n}, v_{n}), (u_{n}, v_{n})\rangle+o_{n}(1) \\
&=\int_{\R^{N}\setminus \Lambda_{\e}} \left[\frac{1}{2}(u_{n} H_{u}(\e x, u_{n}, v_{n})+v_{n} H_{v}(\e x, u_{n}, v_{n}))-H(\e x, u_{n}, v_{n})\right] \, dx \\
&+\int_{\Lambda_{\e}} \left[\frac{1}{2}(u_{n} Q_{u}(u_{n}, v_{n})+v_{n} Q_{v}(u_{n}, v_{n}))-Q(u_{n}, v_{n})\right] \, dx \\
&+\frac{s}{N}\int_{\Lambda_{\e}} K(u_{n}, v_{n}) \, dx+o_{n}(1) \\
&\geq \frac{s}{N}\int_{\Lambda_{\e}} K(u_{n}, v_{n}) \, dx+o_{n}(1) \\
&\geq \frac{s}{N} \int_{\Lambda_{\e}} \psi_{\rho} K(u_{n}, v_{n}) \, dx+o_{n}(1).
\end{align*}
Then, by using \eqref{CML} and taking the limit as $n\rightarrow \infty$, we find
\begin{align*}
c\geq \frac{s}{N} \sum_{\{i\in I: x_{i}\in \Lambda_{\e}\}} \psi_{\rho}(x_{i}) \nu_{i} =\frac{s}{N} \sum_{\{i\in I: x_{i}\in \Lambda_{\e}\}} \nu_{i} \geq \frac{s}{N} \widetilde{S}^{\frac{N}{2s}}_{K},
\end{align*}
which gives a contradiction. This ends the proof of \eqref{ioooo}.

\end{proof}

\noindent
In the next lemma we establish some useful properties of the Nehari manifold $\N_{\e}$.
\begin{lemma}\label{lemma2.2}
There exist positive constants $a_{1}$, $c$ such that, for each $a\in (0, a_{1})$, $(u,v)\in \N_{\e}$, there hold
\begin{equation}\label{5}
\int_{\Lambda_{\e}} (pQ(u, v)+K(u, v)) \, dx\geq c
\end{equation}
and 
\begin{equation}\label{6}
\int_{\R^{N}\setminus \Lambda_{\e}} (V(\e x) u^{2} + W(\e x) v^{2}) \, dx\leq 2 \int_{\Lambda_{\e}} (p Q(u, v)+K(u, v)) \, dx. 
\end{equation}
\end{lemma}

\begin{proof}
By using \eqref{2.6A}, $(Q2)$, $(K2)$ and applying Theorem \ref{Sembedding}, we can see that for any $(u, v)\in \N_{\e}$ it holds
\begin{align*}
\|(u, v)\|^{2}_{\e}&=\int_{\Lambda_{\e}} (u Q_{u} + v Q_{v}+\frac{1}{2^{*}_{s}}(u K_{u}+v K_{v})) \, dx+ \int_{\R^{N}\setminus \Lambda_{\e}} (u H_{u} + vH_{v}) \, dx \\
&\leq C \int_{\Lambda_{\e}} (|u|^{p}+|v|^{p})+(|u|^{2^{*}_{s}}+|v|^{2^{*}_{s}}) \, dx+\frac{1}{2} \int_{\R^{N}} (V(\e x) u^{2} + W(\e x) v^{2}) \, dx \\
&\leq C \left[\|(u, v)\|^{p}_{\e}+\|(u, v)\|^{2^{*}_{s}}_{\e}\right]+\frac{1}{2} \|(u, v)\|^{2}_{\e}
\end{align*} 
which implies that there is $c_{2}>0$ such that
\begin{equation*}
\|(u, v)\|_{\e} \geq c_{2} \, \mbox{ for any } (u, v)\in \N_{\e}.
\end{equation*}
In view of \eqref{2.1} and \eqref{2.6A}, we obtain
\begin{align*}
c_{2}^{2} \leq \|(u,v)\|_{\e}^{2} &= \int_{\Lambda_{\e}} (u Q_{u} + v Q_{v}+\frac{1}{2^{*}_{s}}(u K_{u}+v K_{v})) \, dx+ \int_{\R^{N}\setminus \Lambda_{\e}} (u H_{u} + vH_{v}) \, dx\\
&\leq  \int_{\Lambda_{\e}} (p Q(u, v)+K(u, v)) \, dx+ \frac{1}{2} \int_{\R^{N}\setminus \Lambda_{\e}} (V(\e x) u^{2} + W(\e x) v^{2}) \, dx,  
\end{align*}
which gives
\begin{equation*}
\frac{c_{2}^{2}}{2}\leq \frac{1}{2} \|(u,v)\|_{\e}^{2} \leq  \int_{\Lambda_{\e}} (p Q(u, v)+K(u, v)) \, dx. 
\end{equation*}
Therefore, \eqref{5} holds with $c= \frac{c_{2}^{2}}{2}$. \\
Now, taking into account $(u,v)\in \N_{\e}$, \eqref{2.1} and \eqref{2.6A}, we get
\begin{align*}
\int_{\R^{N}\setminus \Lambda_{\e}} &(V(\e x) u^{2} + W(\e x) v^{2}) \, dx \leq \|(u, v)\|^{2}_{\e} \\
&\leq \int_{\R^{N}} (uH_{u} + vH_{v}) \, dx \\
&=\int_{\Lambda_{\e}} (p Q(u, v)+K(u, v)) \,dx+\int_{\R^{N}\setminus \Lambda_{\e}} (uH_{u} + vH_{v}) \, dx \\
&\leq \int_{\Lambda_{\e}} (p Q(u, v)+K(u, v)) \,dx+\frac{1}{2} \int_{\R^{N}\setminus \Lambda_{\e}} (V(\e x)u^{2} + W(\e x) v^{2} ) \, dx 
\end{align*}
which implies that \eqref{6} is satisfied.  
\end{proof}

\begin{lemma}\label{lemma2.3}
Let $\phi_{\e}: X_{\e} \rightarrow \R$ be given by
\begin{equation*}
\phi_{\e}(u, v):= \|(u,v)\|_{\e}^{2} - \int (uH_{u}(\e x, u, v) + vH_{v}(\e x, u,v)) \, dx. 
\end{equation*}
Then, there exist $a_{2}, b>0$ such that, for each $a\in (0, a_{2})$, 
\begin{equation}\label{7}
\langle\phi_{\e}'(u,v),(u,v)\rangle\leq -b<0 \, \mbox{ for each } (u, v)\in \N_{\e}. 
\end{equation}
\end{lemma}

\begin{proof}
Let $a$ sufficiently small so that Lemma \ref{lemma2.2} holds, and fix $(u,v)\in \N_{\e}$.
By using the definition of $H$, \eqref{2.1} and \eqref{2.11}, we can see that
\begin{align*}
&\langle\phi_{\e}'(u,v),(u,v)\rangle \\
&= \int_{\Lambda_{\e}} (uQ_{u} + vQ_{v}+\frac{1}{2^{*}_{s}}(u K_{u}+v K_{v})) \, dx-\int_{\Lambda_{\e}} (u^{2}Q_{uu} + v^{2}Q_{vv} +2uvQ_{uv}) \, dx\\
&-\frac{1}{2^{*}_{s}}\int_{\Lambda_{\e}} (u^{2}K_{uu} + v^{2}K_{vv} +2uvK_{uv}) \, dx+\int_{\R^{N}\setminus \Lambda_{\e}} (D_{1}-D_{2}) \, dx \\
&=-p(p-2) \int_{\Lambda_{\e}} Q(u,v)\, dx-(2^{*}_{s}-2)\int_{\Lambda_{\e}} K(u, v) \, dx + \int_{\R^{N}\setminus \Lambda_{\e}} (D_{1}- D_{2}) \, dx
\end{align*}
where 
\begin{equation*}
D_{1}:= (uH_{u}+ vH_{v}) \, \mbox{ and } \, D_{2}:= (u^{2}H_{uu} + v^{2}H_{vv} +2uvH_{uv}). 
\end{equation*}
Since $2<p<2^{*}_{s}$, we can see that
\begin{align}\label{MIOO}
\langle\phi_{\e}'(u,v),(u,v)\rangle \leq -p(p-2) \int_{\Lambda_{\e}} (pQ(u,v)+K(u, v))\, dx+ \int_{\R^{N}\setminus \Lambda_{\e}} (D_{1}- D_{2}) \, dx.
\end{align}
Then we can proceed as in the proof of Lemma $4.2$ in \cite{A5} to prove that, for any $a>0$ sufficiently small 
\begin{equation}\label{MIOOO}
\int_{\R^{N}\setminus \Lambda_{\e}} |D_{1}|+|D_{2}|\, dx\leq \frac{p-2}{4} \int_{\R^{N}\setminus \Lambda_{\e}} V(\e x)u^{2}+W(\e x) v^{2} \,dx.
\end{equation}
Putting together \eqref{MIOO} and \eqref{MIOOO}, and by applying Lemma \ref{lemma2.2}, we have 
\begin{equation*}
\langle\phi_{\e}'(u,v),(u,v)\rangle\leq -\frac{(p-2)}{2} \int_{\Lambda_{\e}} (pQ(u,v)+K(u, v))  \, dx\leq -\frac{(p-2)}{2}c = -b<0. 
\end{equation*}
 
\end{proof}

\noindent
At this point, we are able to deduce the main compactness result of this section.
\begin{proposition}\label{prop1}
The functional $\J_{\e}$ restricted to $\N_{\e}$ satisfies $(PS)_{c}$ for each $c\in \R$. 
\end{proposition}

\begin{proof}
Let $\{(u_{n}, v_{n})\}\subset \N_{\e}$ be such that
\begin{equation*}
\J_{\e}(u_{n}, v_{n})\rightarrow c \, \mbox{ and } \, \|\J_{\e}'(u_{n}, v_{n})\|_{*}=o_{n}(1), 
\end{equation*}
where $o_{n}(1)$ goes to zero when $n\rightarrow \infty$. Then, there exists $\{\lambda_{n}\}\subset \R$ satisfying
\begin{equation*}
\J_{\e}'(u_{n}, v_{n})= \lambda_{n} \phi_{\e}'(u_{n}, v_{n})+ o_{n}(1), 
\end{equation*}
with $\phi_{\e}$ as in Lemma \ref{lemma2.3}. Due to the fact that $(u_{n}, v_{n})\in \N_{\e}$, we obtain
\begin{equation}\label{Ter}
0= \langle\J_{\e}'(u_{n}, v_{n}),(u_{n}, v_{n})\rangle= \lambda_{n} \langle\phi_{\e}'(u_{n}, v_{n}),(u_{n}, v_{n})\rangle + o_{n}(1) \|(u_{n}, v_{n})\|_{\e}. 
\end{equation}
Standard calculations show that $\{(u_{n}, v_{n})\}$ is bounded in $\X_{\e}$. In view of Lemma \ref{lemma2.3}, we may assume that $\langle\phi_{\e}'(u_{n}, v_{n}),(u_{n}, v_{n})\rangle\rightarrow \ell<0$. By using \eqref{Ter}, we can deduce that $\lambda_{n}\rightarrow 0$ and as a consequence $\J_{\e}'(u_{n}, v_{n})\rightarrow 0$ in the dual space of $\X_{\e}$. By applying Lemma \ref{lemma3.2-1}, we can infer that $\{(u_{n}, v_{n})\}$ admits a convergent subsequence. 
\end{proof}

\noindent
Arguing as in the proof of the above proposition, it is easy to see that
\begin{corollary}\label{cor3.5}
The critical points of $\J_{\e}$ constrained to $\N_{\e}$ are critical points of $\J_{\e}$ in $\X_{\e}$.
\end{corollary}

\section{Multiplicity of solutions to \eqref{P''}}

\noindent
This section is devoted to the study of the multiplicity of solutions for the system \eqref{P''}.

For this reason, let $\delta>0$ such that
$$
M_{\delta}=\{x\in \R^{N}: dist(x, M)\leq \delta\}\subset \Lambda,
$$
and take $\psi\in C^{\infty}_{0}(\R_{+}, [0, 1])$ is a function satisfying $\psi(t)=1$ if $0\leq t\leq \frac{\delta}{2}$ and $\psi(t)=0$ if $t\geq \delta$.

For any $y\in M$, we define
$$
\Psi_{i, \e, y}(x)=\psi(|\e x-y|) w_{i}\left(\frac{\e x-y}{\e}\right) \quad i=1, 2,
$$
and denote by $t_{\e}>0$ the unique positive number such that 
$$
\max_{t\geq 0} \J_{\e}(t \Psi_{1, \e, y}, t \Psi_{2, \e, y})=\J_{\e}(t_{\e} \Psi_{1,\e, y}, t_{\e} \Psi_{2,\e, y}),
$$
where  $(w_{1}, w_{2})\in \X_{0}$ is a solution for \eqref{P0} with $\xi=x_{0}$ (such solution exists in view of Theorem \ref{prop4.3}) such that $\J_{x_{0}}(w_{1}, w_{2})=C(x_{0})$.\\
Let $\Phi_{\e}: M\rightarrow \N_{\e}$ be given by
$$
\Phi_{\e}(y)=(t_{\e} \Psi_{1, \e, y}, t_{\e} \Psi_{2, \e, y}).
$$
Then we have the following result.
\begin{lemma}\label{lemma3.4FS}
The functional $\Phi_{\e}$ satisfies the following limit
\begin{equation}\label{3.2}
\lim_{\e\rightarrow 0} \J_{\e}(\Phi_{\e}(y))=C^{*} \mbox{ uniformly in } y\in M.
\end{equation}
\end{lemma}
\begin{proof}
Assume by contradiction that there there exist $\delta_{0}>0$, $(y_{n})\subset M$ and $\e_{n}\rightarrow 0$ such that 
\begin{equation}\label{4.41}
|\J_{\e_{n}}(\Phi_{\e_{n}}(y_{n}))-C^{*}|\geq \delta_{0}.
\end{equation}
We aim to prove that $\lim_{n\rightarrow \infty}t_{\e_{n}}<\infty$.
Let us observe that by using the change of variable $z=\frac{\e_{n}x-y_{n}}{\e_{n}}$, if $z\in B_{\frac{\delta}{\e_{n}}}(0)$, it follows that $\e_{n} z\in B_{\delta}(0)$ and $\e_{n} z+y_{n}\in B_{\delta}(y_{n})\subset M_{\delta}\subset \Lambda$. 

Then, recalling that $H=Q+\frac{1}{2^{*}_{s}} K$ on $\Lambda$ and $\psi(t)=0$ for $t\geq \delta$, we have
\begin{align}\label{HeZou}
&\J_{\e}(\Phi_{\e_{n}}(y_{n}))\nonumber \\
&=\frac{t_{\e_{n}}^{2}}{2}\int_{\R^{N}} |(-\Delta)^{\frac{s}{2}}(\psi(|\e_{n} z|)w_{1}(z))|^{2}\, dz+\frac{t_{\e_{n}}^{2}}{2}\int_{\R^{N}} |(-\Delta)^{\frac{s}{2}}(\psi(|\e_{n} z|)w_{2}(z))|^{2}\, dz \nonumber\\
&+\frac{t_{\e_{n}}^{2}}{2}\int_{\R^{N}} V(\e_{n} z+y_{n}) (\psi(|\e_{n} z|) w_{1}(z))^{2}\, dz\nonumber \\
&+\frac{t_{\e_{n}}^{2}}{2}\int_{\R^{N}} W(\e_{n} z+y_{n}) (\psi(|\e_{n} z|) w_{2}(z))^{2}\, dz \nonumber\\
&-\int_{\R^{N}} Q(t_{\e_{n}}\psi(|\e_{n} z|)w_{1}(z), t_{\e_{n}}\psi(|\e_{n} z|)w_{2}(z)) \, dz \nonumber \\
&-\frac{1}{2^{*}_{s}}\int_{\R^{N}} K(t_{\e_{n}}\psi(|\e_{n} z|)w_{1}(z), t_{\e_{n}}\psi(|\e_{n} z|)w_{2}(z)) \, dz.
\end{align}
Assume by contradiction that $t_{\e_{n}}\rightarrow \infty$. From the definition of $t_{\e_{n}}$, $(Q1)$, $(K1)$ and \eqref{2.1}, we get
\begin{align}\label{3.9}
\|(\Psi_{1,\e_{n}, y_{n}}, \Psi_{2,\e_{n}, y_{n}})\|^{2}_{\e_{n}}&=p t_{\e_{n}}^{p-2}\int_{\R^{N}} Q(\psi(|\e_{n} z|)w_{1}(z), \psi(|\e_{n} z|)w_{2}(z)) \, dz \nonumber\\
&+t_{\e_{n}}^{2^{*}_{s}-2} \int_{\R^{N}} K(\psi(|\e_{n} z|)w_{1}(z), \psi(|\e_{n} z|)w_{2}(z)) \, dz
\end{align}
Recalling that $\psi=1$ in $B_{\frac{\delta}{2}}(0)$ and $B_{\frac{\delta}{2}}(0)\subset B_{\frac{\delta}{2\e_{n}}}(0)$ for $n$ big enough, and $w_{1}$, $w_{2}$ are continuous and positive in $\R^{N}$ (see Remark \ref{remark}), we can see that
\begin{align}\label{3.10}
\|(\Psi_{1,\e_{n}, y_{n}}, \Psi_{2,\e_{n}, y_{n}})\|^{2}_{\e_{n}}&\geq  t_{\e_{n}}^{2^{*}_{s}-2} \int_{B_{\frac{\delta}{2}}(0)} K(w_{1}(z),w_{2}(z)) \, dz\nonumber \\
&\geq |B_{\frac{\delta}{2}}(0)| \min_{x\in B_{\frac{\delta}{2}}}K(w_{1}(x), w_{2}(x)) t_{\e_{n}}^{2^{*}_{s}-2}.
\end{align}
By passing to the limit as $n\rightarrow \infty$ in (\ref{3.10}) we can deduce that
$$
\lim_{n\rightarrow \infty} \|(\Psi_{1,\e_{n}, y_{n}}, \Psi_{2,\e_{n}, y_{n}})\|^{2}_{\e_{n}}=\infty,
$$
which is impossible due to the fact that Lemma $5$ in \cite{PP} and the Dominated Convergence Theorem imply
$$
\lim_{n\rightarrow \infty} \|(\Psi_{1,\e_{n}, y_{n}}, \Psi_{2,\e_{n}, y_{n}})\|^{2}_{\e_{n}}=\|(w_{1}, w_{2})\|^{2}_{x_{0}}\in (0, \infty).
$$
Thus, $\{t_{\e_{n}}\}$ is bounded, and we can assume that $t_{\e_{n}}\rightarrow t_{0}\geq 0$. Clearly, if $t_{0}=0$, by boundedness of $\|(\Psi_{1,\e_{n}, y_{n}}, \Psi_{2,\e_{n}, y_{n}})\|^{2}_{\e_{n}}$, the growth assumptions on $Q$ and $K$, and (\ref{3.9}), we can see that $\|(\Psi_{1,\e_{n}, y_{n}}, \Psi_{2,\e_{n}, y_{n}})\|^{2}_{\e_{n}}\rightarrow 0$, which gives a contradiction. Hence, $t_{0}>0$.

Now, by using $(Q2)$, $(K3)$ and the Dominated Convergence Theorem, we have as $n\rightarrow \infty$
$$
\int_{\R^{N}} Q(\Psi_{1, \e_{n}, y_{n}}, \Psi_{2, \e_{n}, y_{n}}) dx\rightarrow \int_{\R^{N}} Q(w_{1}, w_{2}) \, dx
$$
and
$$
\int_{\R^{N}} K(\Psi_{1, \e_{n}, y_{n}}, \Psi_{2, \e_{n}, y_{n}}) dx\rightarrow \int_{\R^{N}} K(w_{1}, w_{2}) \, dx.
$$
Taking the limit as $n\rightarrow \infty$ in (\ref{3.9}),  we obtain
$$
\|(w_{1}, w_{2})\|^{2}_{x_{0}}=p t_{0}^{p-2} \int_{\R^{N}} Q(w_{1}, w_{2}) \, dx+t_{0}^{2^{*}_{s}-2}\int_{\R^{N}} K(w_{1}, w_{2}) \, dx .
$$ 
Since $(w_{1}, w_{2})\in \mathcal{N}_{x_{0}}$, we deduce that $t_{0}=1$. Moreover, from \eqref{HeZou}, we get
$$
\lim_{n\rightarrow \infty} \J_{\e}(\Phi_{\e_{n}}(y_{n}))=\J_{x_{0}}(w_{1}, w_{2})=C^{*},
$$
which contradicts (\ref{4.41}).

\end{proof}

\noindent
Now, we take $\rho=\rho(\delta)>0$ such that $M_{\delta}\subset B_{\rho}$, and we consider $\varUpsilon: \R^{N}\rightarrow \R^{N}$ defined by setting
 \begin{equation*}
 \varUpsilon(x)=
 \left\{
 \begin{array}{ll}
 x &\mbox{ if } |x|<\rho \\
 \frac{\rho x}{|x|} &\mbox{ if } |x|\geq \rho.
  \end{array}
 \right.
 \end{equation*}
We define the barycenter map $\beta_{\e}: \N_{\e}\rightarrow \R^{N}$ by setting
\begin{align*}
\beta_{\e}(u, v)=\frac{\int_{\R^{N}} \varUpsilon(\e x)(u^{2}(x)+v^{2}(x)) \,dx}{\int_{\R^{N}} u^{2}(x)+v^{2}(x) \, dx}.
\end{align*}
Since $M\subset B_{\rho}$, by the definition of $\varUpsilon$ and the Dominated Convergence Theorem, we can proceed as in \cite{A5} to see that 
\begin{lemma}\label{lemma3.5FS}
The functional $\Phi_{\e}$ verifies the following limit
\begin{equation}\label{3.3}
\lim_{\e \rightarrow 0} \beta_{\e}(\Phi_{\e}(y))=y \mbox{ uniformly in } y\in M.
\end{equation}
\end{lemma}

\noindent
The next result will be fundamental to implement the barycenter machinery. Moreover, it allows us to prove that the solutions of the modified problem \eqref{P''} are solutions of the original problem \eqref{P'}. 
\begin{lemma}\label{lem3.1}
Let $\e_{n}\rightarrow 0^{+}$ and $\{(u_{n}, v_{n})\}\subset \N_{\e_{n}}$ be such that $\J_{\e_{n}}(u_{n}, v_{n})\rightarrow C^{*}$. Then there exists $\{\tilde{y}_{n}\}\subset \R^{N}$ such that the translated sequence 
\begin{equation*}
(\tilde{u}_{n}(x), \tilde{v}_{n}(x)):=(u_{n}(x+ \tilde{y}_{n}), v_{n}(x+\tilde{y}_{n}))
\end{equation*}
has a subsequence which converges in $\X_{0}$. Moreover, up to a subsequence, $\{y_{n}\}:=\{\e_{n}\tilde{y}_{n}\}$ is such that $y_{n}\rightarrow y\in M$. 
\end{lemma}

\begin{proof}
Since $\langle \J'_{\e_{n}}(u_{n}, v_{n}), (u_{n}, v_{n}) \rangle=0$ and $\J_{\e_{n}}(u_{n}, v_{n})\rightarrow C^{*}$, we can see that $\{(u_{n}, v_{n})\}$ is bounded. 
We note that $\|(u_{n}, v_{n})\|_{\e_{n}} \not\rightarrow 0$ since $C^{*}>0$. Thus, arguing as in \cite{ASy}, there exist a sequence $\{\tilde{y}_{n}\}\subset \R^{N}$ and constants $R, \gamma>0$ such that
\begin{equation*}
\liminf_{n\rightarrow \infty}\int_{B_{R}(y_{n})} (|u_{n}|^{2}+|v_{n}|^{2}) dx\geq \gamma,
\end{equation*}
which implies that
\begin{equation*}
(\tilde{u}_{n}, \tilde{v}_{n})\rightharpoonup (\tilde{u}, \tilde{v}) \mbox{ weakly in } \X_{0},  
\end{equation*}
where $(\tilde{u}_{n}(x), \tilde{v}_{n}(x)):=(u_{n}(x+ \tilde{y}_{n}), v_{n}(x+\tilde{y}_{n}))$ and $(\tilde{u}, \tilde{v})\neq (0,0)$. \\
Let $\{t_{n}\}\subset (0, +\infty)$ be such that $(\hat{u}_{n}, \hat{v}_{n}):=(t_{n}\tilde{u}_{n}, t_{n}\tilde{v}_{n})\in \N_{x_{0}}$, and set $y_{n}:=\e_{n}\tilde{y}_{n}$.  \\
From the definition of $H$ and $(H3)$, we can see that
\begin{align*}
C^{*}&\leq \J_{x_{0}}(\hat{u}_{n}, \hat{v}_{n})\\
&= \frac{t_{n}^{2}}{2} \|(u_{n}, v_{n})\|^{2}_{x_{0}} - \int_{\R^{N}} Q(t_{n} u_{n}, t_{n} v_{n})\, dx-\frac{1}{2^{*}_{s}} \int_{\R^{N}} K(t_{n} u_{n}, t_{n} v_{n})\, dx \\
&\leq \frac{t_{n}^{2}}{2} \|(u_{n}, v_{n})\|^{2}_{\e_{n}} - \int_{\R^{N}} H(\e x, t_{n} u_{n}, t_{n} v_{n})\, dx \\
&=\J_{\e_{n}}(t_{n}u_{n}, t_{n}v_{n}) \leq \J_{\e_{n}}(u_{n}, v_{n})= C^{*}+ o_{n}(1),
\end{align*}
which gives $\J_{x_{0}}(\hat{u}_{n}, \hat{v}_{n})\rightarrow C^{*}$. \\
Now, the sequence $\{t_{n}\}$ is bounded since $\{(\tilde{u}_{n}, \tilde{v}_{n})\}$ and $\{(\hat{u}_{n}, \hat{v}_{n})\}$ are bounded and $(\tilde{u}_{n}, \tilde{v}_{n})\not \rightarrow 0$. Therefore, up to a subsequence, $t_{n}\rightarrow t_{0}\geq 0$. Indeed $t_{0}>0$. Otherwise, if $t_{0}=0$, from the boundedness of $\{(\tilde{u}_{n}, \tilde{v}_{n})\}$, we get $(\hat{u}_{n}, \hat{v}_{n})= t_{n}(\tilde{u}_{n}, \tilde{v}_{n}) \rightarrow (0,0)$, that is $\J_{x_{0}}(\hat{u}_{n}, \hat{v}_{n})\rightarrow 0$ in contrast with the fact $C^{*}>0$. Thus $t_{0}>0$, and up to a subsequence, we have $(\hat{u}_{n}, \hat{v}_{n})\rightharpoonup t_{0}(\tilde{u}, \tilde{v})= (\hat{u}, \hat{v})\neq 0$ weakly in $\X_{0}$. 
Hence, it holds
\begin{equation*}
\J_{x_{0}}(\hat{u}_{n}, \hat{v}_{n})\rightarrow C^{*} \quad \mbox{ and } \quad (\hat{u}_{n}, \hat{v}_{n})\rightharpoonup (\hat{u}, \hat{v}) \mbox{ weakly in } \X_{0}.
\end{equation*}
From Theorem \ref{prop4.3}, we deduce that $(\hat{u}_{n}, \hat{v}_{n})\rightarrow (\hat{u}, \hat{v})$ in $\X_{0}$, that is $(\tilde{u}_{n}, \tilde{v}_{n})\rightarrow (\tilde{u}, \tilde{v})$ in $\X_{0}$. \\
Now, we show that $\{y_{n}\}$ has a subsequence such that $y_{n}\rightarrow y\in M$. 
We argue by contradiction, and we assume that, up to a subsequence, $|y_{n}|\rightarrow +\infty$. \\
Since $(u_{n}, v_{n})\in \N_{\e_{n}}$, we get
\begin{align*}
&\int_{\R^{N}} |(-\Delta)^{\frac{s}{2}}\tilde{u}_{n}|^{2}+|(-\Delta)^{\frac{s}{2}}\tilde{v}_{n}|^{2}+V(\e_{n}x+y_{n})|\tilde{u}_{n}|^{2}+ W(\e_{n}x+y_{n})|\tilde{v}_{n}|^{2} dx \\
&=\int_{\R^{N}} \tilde{u}_{n}H_{u}(\e_{n} x+y_{n},\tilde{u}_{n} ,\tilde{v}_{n})+\tilde{v}_{n}H_{v}(\e_{n} x+y_{n},\tilde{u}_{n} ,\tilde{v}_{n})\, dx.
\end{align*}
Take $R>0$ such that $\Lambda \subset B_{R}(0)$. Since we may assume that $|y_{n}|>2R$, for any $x\in B_{R/\e_{n}}(0)$ we get $|\e_{n} x+y_{n}|\geq |y_{n}|-|\e_{n} x|>R$.
Then, recalling the definition of $H$ and by using \eqref{2.6A}, the strong convergence of $(\tilde{u}_{n}, \tilde{v}_{n})$ and $|\R^{N}\setminus B_{R/\e_{n}}(0)|\rightarrow 0$ as $n\rightarrow \infty$, it follows that
\begin{align*}
&\int_{\R^{N}} \tilde{u}_{n}H_{u}(\e_{n} x+y_{n},\tilde{u}_{n} ,\tilde{v}_{n})+\tilde{v}_{n}H_{v}(\e_{n} x+y_{n},\tilde{u}_{n} ,\tilde{v}_{n})\, dx \\
&\leq \frac{1}{2} \int_{B_{R/\e_{n}}(0)} V(\e_{n}x+y_{n})|\tilde{u}_{n}|^{2}+ W(\e_{n}x+y_{n})|\tilde{v}_{n}|^{2} dx \\
&+\int_{\R^{N}\setminus B_{R/\e_{n}}(0)} \tilde{u}_{n}H_{u}(\e_{n} x+y_{n},\tilde{u}_{n} ,\tilde{v}_{n})+\tilde{v}_{n}H_{v}(\e_{n} x+y_{n},\tilde{u}_{n} ,\tilde{v}_{n})\, dx \\
&=\frac{1}{2} \int_{B_{R/\e_{n}}(0)} V(\e_{n}x+y_{n})|\tilde{u}_{n}|^{2}+ W(\e_{n}x+y_{n})|\tilde{v}_{n}|^{2} dx+o_{n}(1).
\end{align*}
By using $(H3)$, we obtain
$$
\left(1-\frac{1}{2}\right)\|(\tilde{u}_{n}, \tilde{v}_{n})\|^{2}_{x_{0}}=o_{n}(1),
$$
which gives a contradiction because of $(\hat{u}_{n}, \hat{v}_{n})\rightarrow (\hat{u}, \hat{v})\neq 0$.
Thus $\{y_{n}\}$ is bounded and, up to a subsequence, we may assume that $y_{n}\rightarrow y$. If $y\notin \overline{\Lambda}$, then there exists $r>0$ such that $y_{n}\in B_{r/2}(u)\subset \R^{N}\setminus \overline{\Lambda}$ for any $n$ large enough. Reasoning as before, we get a contradiction. Hence $y\in \overline{\Lambda}$. \\
Now, we prove that $y\in M$. Taking into account Lemma \ref{C00}, it is enough to prove that $C(y)=C^{*}$. Assume by contradiction that $C(y)<C^{*}$.
Since $(\hat{u}_{n}, \hat{v}_{n})\rightarrow (\hat{u}, \hat{v})$ strongly in $\X_{0}$, by Fatou Lemma we have
\begin{align*}
C^{*}<C(y)&= \J_{y}(\hat{u}, \hat{v}) \\
&= \liminf_{n\rightarrow \infty} \Bigl\{ \frac{1}{2}\left(\int_{\R^{N}} |(-\Delta)^{\frac{s}{2}} \hat{u}_{n}|^{2}+ |(-\Delta)^{\frac{s}{2}} \hat{v}_{n}|^{2} dx\right) \\
&- \int_{\R^{N}} \left(Q(\hat{u}_{n}, \hat{v}_{n})+\frac{1}{2^{*}_{s}} K(\hat{u}_{n}, \hat{v}_{n})\right)\, dx \nonumber\\
&+ \frac{1}{2} \int_{\R^{N}} (V(\e_{n}x + y_{n})|\hat{u}_{n}|^{2} + W(\e_{n}x+y_{n})|\hat{v}_{n}|^{2})\, dx\Bigr\} \nonumber \\
&\leq \liminf_{n\rightarrow \infty} \J_{\e_{n}}(t_{n}u_{n}, t_{n}v_{n}) \leq \liminf_{n\rightarrow \infty} \J_{\e_{n}} (u_{n}, v_{n})=C^{*}
\end{align*}
which is impossible. Then $C(y)=C^{*}$ and this ends the proof of lemma. 
\end{proof}

\noindent
Now, we consider a subset $\widetilde{\N}_{\e}$ of $\N_{\e}$ by taking a function $h:\R_{+}\rightarrow \R_{+}$ such that $h(\e)\rightarrow 0$ as $\e \rightarrow 0$, and setting
$$
\widetilde{\N}_{\e}=\{(u, v)\in \N_{\e}: \J_{\e}(u)\leq C^{*}+h(\e)\}.
$$
Fixed $y\in M$, we conclude from Lemma \ref{lemma3.4FS} that $h(\e)=|\J_{\e}(\Phi_{\e}(y))-C^{*}|\rightarrow 0$ as $\e \rightarrow 0$. Hence $\Phi_{\e}(y)\in \widetilde{\N}_{\e}$, and $\widetilde{\N}_{\e}\neq \emptyset$ for any $\e>0$. Moreover, as in \cite{A5}, we can prove the following lemma.
\begin{lemma}\label{lemma3.7FS}
$$
\lim_{\e \rightarrow 0} \sup_{(u, v)\in \widetilde{\mathcal{N}}_{\e}} dist(\beta_{\e}(u, v), M_{\delta})=0.
$$
\end{lemma}

\noindent
At the end of this section, we give the proof of the following multiplicity result to \eqref{P''}.
\begin{theorem}\label{thm3.1AFF}
For any $\delta>0$ satisfying $M_{\delta}\subset M$, there exists $\e_{\delta}>0$ such that for any $\e\in (0, \e_{\delta})$, problem \eqref{P''} has at least $cat_{M_{\delta}}(M)$ positive solutions.
\end{theorem}
\begin{proof}
Fix $\delta>0$ such that $M_{\delta}\subset M$. By using Lemma \ref{lemma3.4FS}, Lemma \ref{lemma3.5FS} and Lemma \ref{lemma3.7FS}, we can find $\e_{\delta}>0$ such that for any $\e\in (0, \e_{\delta})$, the following diagram
$$
M \stackrel{\Phi_{\e}}{\rightarrow}  \widetilde{\N}_{\e} \stackrel{\beta_{\e}}{\rightarrow} M_{\delta}
$$
is well-defined and $\beta_{\e}\circ \Phi_{\e}$ is  homotopically equivalent to the inclusion map $\iota: M\rightarrow M_{\delta}$. 
Since $C^{*}=C(x_{0})<\frac{s}{N} \tilde{S}^{\frac{N}{2s}}_{K}$, we can use the definition of $\widetilde{\N}_{\e}$ and taking $\e_{\delta}$ sufficiently small, we may assume that $\J_{\e}$ verifies the Palais-Smale condition in $\widetilde{\N}_{\e}$ (see Proposition \ref{prop1}). 
By applying Ljusternik-Schnirelmann theory \cite{W}, we obtain at least $cat_{\widetilde{\N}_{\e}}(\widetilde{\N}_{\e})$ critical points $(u_{i}, v_{i}):=(u^{i}_{\e}, v^{i}_{\e})$ of $\J_{\e}$ restricted to $\N_{\e}$. From the arguments in \cite{BC}, we can see that $cat_{\widetilde{\N}_{\e}}(\widetilde{\N}_{\e})\geq cat_{M_{\delta}}(M)$. 
Then, Corollary \ref{cor3.5} implies that each $(u_{i}, v_{i})$ is a critical point of the unconstrained functional and as a consequence a solution of the problem \eqref{P''}.

\end{proof}

\section{Proof of Theorem \ref{thm1}}

In this  section we give the proof of Theorem \ref{thm1}. We follow the ideas developed in \cite{A5}.
\begin{proof}
Fix $\delta>0$ such that $M_\delta \subset \Lambda$. We aim to show that there exists $\tilde{\e}_{\delta}>0$ such that for any $\e \in (0, \tilde{\e}_{\delta})$ and any solution $u_{\e} \in \widetilde{\N}_{\e}$ of \eqref{P''}, it holds 
\begin{equation}\label{infty}
\|(u_{\e}, v_{\e})\|_{L^{\infty}(\R^{N}\setminus \Lambda_{\e})}<a. 
\end{equation}
We argue by contradiction, and we suppose that there exist $\e_{n}\rightarrow 0$, $(u_{\e_{n}}, v_{\e_{n}})\in \widetilde{\mathcal{N}}_{\e_{n}}$ such that $\J'_{\e_{n}}(u_{\e_{n}}, v_{\e_{n}})=0$ and $\|(u_{\e_{n}}, v_{\e_{n}})\|_{L^{\infty}(\R^{N}\setminus \Lambda_{\e_{n}})}\geq a$. 
Since $\J_{\e_{n}}(u_{\e_{n}}, v_{\e_{n}}) \leq C^{*} + h(\e_{n})$ and $h(\e_{n})\rightarrow 0$, we can proceed as in the  first part of the proof of Lemma \ref{lem3.1}, to deduce that $\J_{\e_{n}}(u_{\e_{n}}, v_{\e_{n}})\rightarrow C^{*}$.
From Lemma \ref{lem3.1}, it follows the existence of a sequence $\{\tilde{y}_{n}\}\subset \R^{N}$ such that $\e_{n}\tilde{y}_{n}\rightarrow y \in M$. \\
Now, we take $r>0$ such that $B_{2r}(y)\subset \Lambda$, so $B_{\frac{r}{\e_{n}}}(\frac{y}{\e_{n}})\subset \Lambda_{\e_{n}}$. Moreover, for any $z\in B_{\frac{r}{\e_{n}}}(\tilde{y}_{n})$ we can see
\begin{equation*}
\left|z - \frac{y}{\e_{n}}\right| \leq |z- \tilde{y}_{n}|+ \left|\tilde{y}_{n} - \frac{y}{\e_{n}}\right|<\frac{2r}{\e_{n}}\, \mbox{ for } n \mbox{ sufficiently large. }
\end{equation*}
Hence, for any $n$ big enough, we have 
$$
\R^{N}\setminus \Lambda_{\e_{n}}\subset \R^{N} \setminus B_{\frac{r}{\e_{n}}}(\tilde{y}_{n}).
$$
Let us denote by $(\tilde{u}_{n}(x), \tilde{v}_{n}(x))=(u_{\e_{n}}(x+ \tilde{y}_{n}), v_{\e_{n}}(x+ \tilde{y}_{n}))$ and we set $\tilde{z}_{n}=\tilde{u}_{n}+\tilde{v}_{n}$. Since $\tilde{z}_{n}\geq 0$ satisfies
$$
(-\Delta)^{s} z+\alpha z\leq C_{0}(z^{p-1}+ z^{2^{*}_{s}-1}) \mbox{ in } \R^{N},
$$
where $\alpha=\min\{V(x_{0}), W(x_{0})\}$ and for some positive constant $C_{0}$ given by $(Q2)$ and $(K3)$, 
we can use a Moser iteration argument (see \cite{Ambrosio, A3, DPMV, HZ}) and the fact that $\{(u_{\e_{n}}, v_{\e_{n}})\}$ is bounded in $\X_{\e_{n}}$, to see that $\tilde{z}_{n}\in L^{\infty}(\R^{N})$, and there exists a constant $K>0$ such that  
$$
\|\tilde{z}_{n}\|_{L^{\infty}(\R^{N})}\leq K \mbox{ for any } n\in \mathbb{N},
$$
and 
$\tilde{u}_{n}\rightarrow u$  and $\tilde{v}_{n}\rightarrow v$ in $L^{q}(\R^{N})$ for any $q\in (2, 2^{*}_{s})$, for some $(u, v)\in L^{q}(\R^{N})$ for any $q\in (2, 2^{*}_{s})$.\\
Moreover, by interpolation, we can see that
$$
H_{u}(\e_{n}x+\e_{n}\tilde{y}_{n}, \tilde{u}_{n}, \tilde{v}_{n})\rightarrow Q_{u}(u, v)+\frac{1}{2^{*}_{s}}K_{u}(u, v)
$$
and 
$$
H_{v}(\e_{n}x+\e_{n}\tilde{y}_{n}, \tilde{u}_{n}, \tilde{v}_{n})\rightarrow Q_{v}(u, v)+\frac{1}{2^{*}_{s}}K_{v}(u, v).
$$ 
in the sense of convergence in $L^{q}(\R^{N})$ for any $q\in (2, 2^{*}_{s})$.\\ 
Since $\tilde{z}_{n}$ is a solution to 
$$
(-\Delta)^{s} \tilde{z}_{n}+\tilde{z}_{n}=\xi_{n} \mbox{ in } \R^{N}
$$
where 
\begin{align*}
\xi_{n}&:=H_{u}(\e_{n}x+\e_{n}\tilde{y}_{n}, \tilde{u}_{n}, \tilde{v}_{n})+H_{v}(\e_{n}x+\e_{n}\tilde{y}_{n}, \tilde{u}_{n}, \tilde{v}_{n}) \\
&-V(\e_{n}x+\e_{n}\tilde{y}_{n}) \tilde{u}_{n}-W(\e_{n}x+\e_{n}\tilde{y}_{n})\tilde{v}_{n}+\tilde{z}_{n},
\end{align*}
and
$$
\xi_{n}\rightarrow Q_{u}(u, v)+Q_{v}(u, v)+\frac{1}{2^{*}_{s}}(K_{u}(u, v)+K_{v}(u, v))-V(y)u-W(y) v+(u+v) \mbox{ in } L^{q}(\R^{N})
$$
for any $q\in [2, 2^{*}_{s})$, there exists $K_{1}>0$ such that 
$$
\|\xi_{n}\|_{L^{\infty}(\R^{N})}\leq K_{1} \mbox{ for any } n\in \mathbb{N}.
$$
As a consequence, $\tilde{z}_{n}(x)=(\mathcal{K}*\xi_{n})(x)=\int_{\R^{N}} \mathcal{K}(x-t) \xi_{n}(t) \, dt$, where $\mathcal{K}$ is the Bessel kernel satisfying the following properties \cite{FQT}:
\begin{compactenum}[$(i)$]
\item $\mathcal{K}$ is positive, radially symmetric and smooth in $\R^{N}\setminus \{0\}$,
\item there is $C>0$ such that $\mathcal{K}(x)\leq \frac{C}{|x|^{N+2s}}$ for any $x\in \R^{N}\setminus \{0\}$,
\item $\mathcal{K}\in L^{q}(\R^{N})$ for any $q\in [1, \frac{N}{N-2s})$.
\end{compactenum}
Then, arguing as in Lemma $2.6$ in \cite{AM}, we can see that 
\begin{equation}\label{AM3}
\tilde{z}_{n}(x)\rightarrow 0 \mbox{ as } |x|\rightarrow \infty
\end{equation}
uniformly in $n\in \mathbb{N}$.\\
Hence, there exists $R>0$ such that 
$$
|(\tilde{u}_{n}(x), \tilde{v}_{n}(x))|<a \mbox{ for all } |x|\geq R, n\in \mathbb{N}.
$$
This together with the definition of $(\tilde{u}_{n}, \tilde{v}_{n})$, yields
$$
|(u_{\e_{n}}(x), v_{\e_{n}}(x))|<a \mbox{ for any } x\in \R^{N}\setminus B_{R}(\tilde{y}_{n}), n\in \mathbb{N}.
$$
 As a consequence, there exists $\nu \in \mathbb{N}$ such that for any $n\geq \nu$ and $\frac{r}{\e_{n}}>R$, it holds 
 $$
 \R^{N}\setminus \Lambda_{\e_{n}}\subset \R^{N} \setminus B_{\frac{r}{\e_{n}}}(\tilde{y}_{n})\subset \R^{N}\setminus B_{R}(\tilde{y}_{n}),
 $$
which gives $|(u_{\e_{n}}(x), v_{\e_{n}}(x))|<a$ for any $x\in \R^{N}\setminus \Lambda_{\e_{n}}$ and $n\geq \nu$, that is a contradiction. \\
Now, let $\bar{\e}_{\delta}$ given in Theorem \ref{thm3.1AFF} and take $\e_{\delta}= \min \{\tilde{\e}_{\delta}, \bar{\e}_{\delta}\}$. Fix $\e \in (0, \e_{\delta})$. By Theorem \ref{thm3.1AFF}, we know that problem \eqref{P''} admits $cat_{M_{\delta}}(M)$ nontrivial solutions $(u_{\e}, v_{\e})$. Due to the fact that $(u_{\e}, v_{\e})\in \widetilde{\mathcal{N}}_{\e}$ satisfies \eqref{infty}, from the definition of $H$ and $\hat{Q}$, it follows that $(u_{\e}, v_{\e})$ is a solution of \eqref{P'}. By using $(Q6)$ and the maximum principle for the fractional Laplacian \cite{CabSir}, we can infer that $u_{\e}, v_{\e}>0$ in $\R^{N}$.
We conclude the proof studying the behavior of the maximum points of $(u_{\e}, v_{\e})$.\\
Let $\e_{n}\rightarrow 0$ and take $\{(u_{\e_{n}}, v_{\e_{n}})\}\subset \X_{\e_{n}}$ be a sequence of solutions to \eqref{P''} as above.
From the definition of $H$ and the assumptions $(Q2)$ and $(K3)$,  we can find $\bar{a}\in (0, a)$ sufficiently small such that
\begin{align}\label{HZnew}
uH_{u}(\e_{n} x, u, v)+vH_{v}(\e_{n} x, u, v)&=uQ_{u}+vQ_{v}+\frac{1}{2^{*}_{s}}(uK_{u}+vK_{v}) \nonumber \\
&\leq \frac{\alpha}{2}(u^{2}+v^{2}) \mbox{ for any } x\in \R^{N}, |(u, v)|\leq \bar{a}.
\end{align}
Arguing as before, we can find $R>0$ such that 
\begin{equation}\label{HZnnew}
\|(u_{\e_{n}}, v_{\e_{n}})\|_{L^{\infty}(B_{R}^{c}(\tilde{y}_{n}))}< \bar{a}.
\end{equation}
Up to a subsequence, we may assume that 
\begin{equation}\label{TV1}
\|(u_{\e_{n}}, v_{\e_{n}})\|_{L^{\infty}(B_{R}(\tilde{y}_{n}))}\geq \bar{a}.
\end{equation}
Otherwise, we can deduce that $\|(u_{\e_{n}}, v_{\e_{n}})\|_{L^{\infty}(\R^{N})}< \bar{a}$, and by using the facts $\langle \J'_{\e_{n}}(u_{\e_{n}}, v_{\e_{n}}),(u_{\e_{n}}, v_{\e_{n}}) \rangle=0$ and \eqref{HZnew}, we obtain
\begin{align*}
\|(u_{\e_{n}}, v_{\e_{n}}) \|^{2}_{\e_{n}} &=\int_{\R^{N}} u_{\e_{n}} H_{u}(\e_{n} x, u_{\e_{n}}, v_{\e_{n}})+v_{\e_{n}} H_{v}(\e_{n} x, u_{\e_{n}}, v_{\e_{n}}) \, dx\\
&\leq \frac{\alpha}{2}\int_{\R^{N}} (u_{\e_{n}}^{2}+ v_{\e_{n}}^{2})\, dx,
\end{align*}
which gives $\| (u_{\e_{n}}, v_{\e_{n}})\|_{\e_{n}}= 0$, that is a contradiction. Then, \eqref{HZnnew} holds.\\
Now, we denote by $x_{n}$ and $\bar{x}_{n}$ the  maximum points of $u_{\e_{n}}$ and $v_{\e_{n}}$ respectively. From \eqref{HZnnew} and \eqref{TV1}, it follows that $x_{n}=\tilde{y}_{n}+p_{n}$ and $\bar{x}_{n}=\tilde{y}_{n}+q_{n}$ for some $p_{n}, q_{n}\in B_{R}(0)$. \\
Set $\hat{u}_{n}(x)=u_{\e_{n}}(x/\e_{n})$ and $\hat{v}_{n}(x)=v_{\e_{n}}(x/\e_{n})$.
Then $\hat{u}_{n}$ and $\hat{v}_{n}$ are solutions to \eqref{P} with maximum points $P_{n}:=\e_{n}\tilde{y}_{n}+\e_{n}p_{n}$ and $Q_{n}:=\e_{n}\tilde{y}_{n}+\e_{n}q_{n}$ respectively. Since $|p_{n}|, |q_{n}|<R$ for all $n\in \mathbb{N}$ and $\e_{n}\tilde{y}_{n}\rightarrow y\in M$, we can deduce that $P_{n}, Q_{n} \rightarrow y$, and by using Lemma \ref{C0}, we can see that
$$
\lim_{n\rightarrow \infty} C(P_{n})=\lim_{n\rightarrow \infty} C(Q_{n})=C(y)=C^{*}=C(x_{0}).
$$
Now, we study the decay  properties of $(\hat{u}_{n}, \hat{v}_{n})$ and we show that \eqref{DEuv} holds. 

Let $\tilde{z}_{n}(x)=\tilde{u}_{n}(x)+\tilde{v}_{n}(x)$.
In view of (\ref{AM3}), we know that $\tilde{z}_{n}\rightarrow 0$ as $|x|\rightarrow \infty$ uniformly in $n$.
On the other hand, taking into account $(Q2)$ and $(K3)$, we have
$$
|H_{u}|+|H_{v}|+\frac{1}{2^{*}_{s}} (|K_{u}|+|K_{v}|)= o(|(u, v)|) \mbox{ as } |(u, v)|\rightarrow 0.
$$
Therefore, by using $(H3)$, $\alpha=\min\{V(x_{0}), W(x_{0})\}$ and $\sqrt{x^{2}+y^{2}}\leq x+y$ for any $x, y\geq 0$, we can find $R_{1}>0$ sufficiently large such that
\begin{align}\label{HZ3}
&(-\Delta)^{s} \tilde{z}_{n}+\frac{\alpha}{2} \tilde{z}_{n}\nonumber \\
&=(-\Delta)^{s} \tilde{u}_{n}+(-\Delta)^{s} \tilde{v}_{n}+V \tilde{u}_{n}+W \tilde{v}_{n}-\left(V \tilde{u}_{n}+W \tilde{v}_{n}-\frac{\alpha}{2} \tilde{z}_{n} \right) \nonumber \\
&=H_{u}(\e_{n}x+\e_{n}\tilde{y}_{n}, \tilde{u}_{n}, \tilde{v}_{n})+H_{v}(\e_{n}x+\e_{n}\tilde{y}_{n}, \tilde{u}_{n}, \tilde{v}_{n}) \nonumber \\
&+\frac{1}{2^{*}_{s}} \left(K_{u}(\e_{n}x+\e_{n}\tilde{y}_{n}, \tilde{u}_{n}, \tilde{v}_{n})+K_{v}(\e_{n}x+\e_{n}\tilde{y}_{n}, \tilde{u}_{n}, \tilde{v}_{n})\right) 
-\left(V \tilde{u}_{n}+W \tilde{v}_{n}-\frac{\alpha}{2} \tilde{z}_{n} \right) \nonumber \\
&\leq H_{u}(\e_{n}x+\e_{n}\tilde{y}_{n}, \tilde{u}_{n}, \tilde{v}_{n})+H_{v}(\e_{n}x+\e_{n}\tilde{y}_{n}, \tilde{u}_{n}, \tilde{v}_{n}) \nonumber \\
&+\frac{1}{2^{*}_{s}} \left(K_{u}(\e_{n}x+\e_{n}\tilde{y}_{n}, \tilde{u}_{n}, \tilde{v}_{n})+K_{v}(\e_{n}x+\e_{n}\tilde{y}_{n}, \tilde{u}_{n}, \tilde{v}_{n})\right) -\frac{\alpha}{2} \tilde{z}_{n} \nonumber \\
&\leq H_{u}(\e_{n}x+\e_{n}\tilde{y}_{n}, \tilde{u}_{n}, \tilde{v}_{n})+H_{v}(\e_{n}x+\e_{n}\tilde{y}_{n}, \tilde{u}_{n}, \tilde{v}_{n}) \nonumber \\
&+\frac{1}{2^{*}_{s}} \left(K_{u}(\e_{n}x+\e_{n}\tilde{y}_{n}, \tilde{u}_{n}, \tilde{v}_{n})+K_{v}(\e_{n}x+\e_{n}\tilde{y}_{n}, \tilde{u}_{n}, \tilde{v}_{n})\right) -\frac{\alpha}{2} |(\tilde{u}_{n}, \tilde{v}_{n})|  \nonumber \\
&\leq 0 \mbox{ in } \R^{N}\setminus B_{R_{1}}. 
\end{align}
In virtue of Lemma $4.3$ in \cite{FQT}, we know that there exists $w$ such that 
\begin{align}\label{HZ1}
0<w(x)\leq \frac{C}{1+|x|^{N+2s}},
\end{align}
and
\begin{align}\label{HZ2}
(-\Delta)^{s} w+\frac{\alpha}{2}w\geq 0 \mbox{ in } \R^{N}\setminus B_{R_{1}} 
\end{align}
for some suitable $R_{2}>0$.
Choose $R_{3}=\max\{R_{1}, R_{2}\}$, and we set 
\begin{align}\label{HZ4}
a=\inf_{B_{R_{3}}} w>0 \mbox{ and } \tilde{w}_{n}=(b+1)w-a\hat{z}_{n}.
\end{align}
where $b=\sup_{n\in \mathbb{N}} \|\tilde{z}_{n}\|_{L^{\infty}(\R^{N})}<\infty$. 
Our goal is to show that 
\begin{equation}\label{HZ5}
\tilde{w}_{n}\geq 0 \mbox{ in } \R^{N}.
\end{equation}
Firstly, we observe that
\begin{align}
&\tilde{w}_{n}\geq ba+w-ba>0 \mbox{ in } B_{R_{3}} \label{HZ0},\\
&(-\Delta)^{s} \tilde{w}_{n}+\frac{\alpha}{2}\tilde{w}_{n}\geq 0 \mbox{ in } \R^{N}\setminus B_{R_{3}} \label{HZ00}.
\end{align}
Now, we argue by contradiction, and we assume that there exists a sequence $\{\bar{x}_{j, n}\}\subset \R^{N}$ such that 
\begin{align}\label{HZ6}
\inf_{x\in \R^{N}} \tilde{w}_{n}(x)=\lim_{j\rightarrow \infty} \tilde{w}_{n}(\bar{x}_{j, n})<0. 
\end{align}
By using (\ref{AM3}) and the definition of $\tilde{w}_{n}$, it is clear that $\tilde{w}_{n}(x)\rightarrow 0$ as $|x|\rightarrow \infty$, uniformly in $n\in \mathbb{N}$. Thus, we can deduce that $\{\bar{x}_{j, n}\}$ is bounded, and, up to subsequence, we may assume that there exists $\bar{x}_{n}\in \R^{N}$ such that $\bar{x}_{j, n}\rightarrow \bar{x}_{n}$ as $j\rightarrow \infty$. 
Thus, from (\ref{HZ6}), we get
\begin{align}\label{HZ7}
\inf_{x\in \R^{N}} \tilde{w}_{n}(x)= \tilde{w}_{n}(\bar{x}_{n})<0.
\end{align}
By using the minimality of $\bar{x}_{n}$ and the representation formula for the fractional Laplacian \cite{DPV}, we can see that 
\begin{align}\label{HZ8}
(-\Delta)^{s}\tilde{w}_{n}(\bar{x}_{n})=\frac{C(N, s)}{2} \int_{\R^{N}} \frac{2\tilde{w}_{n}(\bar{x}_{n})-\tilde{w}_{n}(\bar{x}_{n}+\xi)-\tilde{w}_{n}(\bar{x}_{n}-\xi)}{|\xi|^{N+2s}} d\xi\leq 0.
\end{align}
Taking into account (\ref{HZ0}) and (\ref{HZ6}), we can infer that $\bar{x}_{n}\in \R^{N}\setminus B_{R_{3}}$.
This together with (\ref{HZ7}) and (\ref{HZ8}), yield 
$$
(-\Delta)^{s} \tilde{w}_{n}(\bar{x}_{n})+\frac{\alpha}{2}\tilde{w}_{n}(\bar{x}_{n})<0,
$$
which contradicts (\ref{HZ00}).
Thus (\ref{HZ5}) holds, and by using (\ref{HZ1}) we get
\begin{align*}
\tilde{z}_{n}(x)\leq \frac{\tilde{C}}{1+|x|^{N+2s}} \mbox{ for all } n\in \mathbb{N}, x\in \R^{N},
\end{align*}
for some constant $\tilde{C}>0$. \\
Hence, recalling the definition of $\tilde{z}_{n}$, we get  
\begin{align*}
\hat{u}_{n}(x)&=u_{\e_{n}}\left(\frac{x}{\e_{n}}\right)=\tilde{u}_{n}\left(\frac{x}{\e_{n}}-\tilde{y}_{n}\right) \\
&\leq \frac{\tilde{C}}{1+|\frac{x}{\e_{n}}-\tilde{y}_{\e_{n}}|^{N+2s}} \\
&=\frac{\tilde{C} \e_{n}^{N+2s}}{\e_{n}^{N+2s}+|x- \e_{n} \tilde{y}_{\e_{n}}|^{N+2s}} \\
&\leq \frac{\tilde{C} \e_{n}^{N+2s}}{\e_{n}^{N+2s}+|x-P_{\e_{n}}|^{N+2s}}.
\end{align*}
In similar fashion, we have the estimate for $\hat{v}_{n}$. This ends the proof of \eqref{DEuv}.

\end{proof}

\section{Appendix}

In this section we give the proof of Lemma \ref{CCL}. 
Firstly, we prove some technical lemmata.
\begin{lemma}\label{lem1FS}
For any $\e>0$ there exists $C_{\e}>0$ such that 
\begin{equation}\label{FSinequality}
|K(a+b, c+d)-K(a, b)|\leq \e(|a|^{2^{*}_{s}}+|c|^{2^{*}_{s}})+C_{\e}(|b|^{2^{*}_{s}}+|d|^{2^{*}_{s}}) \mbox{ for any } a, b, c, d\in \R.
\end{equation}
\end{lemma}
\begin{proof}
By using the mean value theorem, we know that there exists $\theta\in (0, 1)$ such that
\begin{equation*}
K(a+b, c+d)-K(a, b)=(\nabla K(a+\theta b, c+\theta d),(c, d)),
\end{equation*}
where $(\cdot, \cdot)$ is the inner product in $\R^{2}$.
By using the following elementary inequality
$$
|x+\theta y|^{2^{*}_{s}-1}\leq C_{2^{*}_{s}-1}(|x|^{2^{*}_{s}-1}+|y|^{2^{*}_{s}-1}) \mbox{ for all } x, y\in \R,
$$
and the fact that $K$ is $2^{*}_{s}$-homogeneous, we can see that
\begin{align}\label{ZIO}
|K(a+b, c+d)-K(a, b)|&\leq C(|a|^{2^{*}_{s}-1}|b|+|a|^{2^{*}_{s}-1}|d|+|c|^{2^{*}_{s}-1}|b| \nonumber \\
&+|c|^{2^{*}_{s}-1}|d|+|b|^{2^{*}_{s}}+|d|^{2^{*}_{s}}+|b|^{2^{*}_{s}-1}|d|+|d|^{2^{*}_{s}-1}|b|). 
\end{align}
By using the Young inequality 
$$
x y\leq \e x^{p}+c_{\e} y^{q} \mbox{ for any } x, y\geq 0, \mbox{ with } p, q\geq 1: \frac{1}{p}+\frac{1}{q}=1
$$
to the right hand side of \eqref{ZIO}, we can obtain \eqref{FSinequality}.
\end{proof}

\begin{lemma}\label{lem8BL}
Let $\mu$ be a measure on $\R^{N}$, and assume that 
\begin{compactenum}[$(i)$]
\item $u_{n}\rightarrow u, v_{n}\rightarrow v$ a.e. in $\R^{N}$; 
\item $\int_{\R^{N}} |u_{n}|^{2^{*}_{s}} d\mu, \int_{\R^{N}} |v_{n}|^{2^{*}_{s}} d\mu\leq C$ for any $n\in \mathbb{N}$.
\end{compactenum}
Then we have
\begin{align}
\int_{\R^{N}} K(u_{n}, v_{n}) d\mu-\int_{\R^{N}} K(u_{n}-u, v_{n}-v) d\mu=\int_{\R^{N}} K(u, v) d\mu+o_{n}(1).
\end{align}
\end{lemma}
\begin{proof}
We follow the arguments in \cite{BL}. Fix $\e>0$ and we denote by $C_{\e}>0$ a constant such that \eqref{FSinequality} holds.
Let us consider the following sequence
$$
\xi_{n}:=|K(u_{n}, v_{n})- K(u_{n}-u, v_{n}-v)- K(u, v)|.
$$
and in view of \eqref{FSinequality} with $a=u_{n}-u$, $b=u$, $c=v_{n}-v$, $d=v$ we can see that
$$
0\leq\xi_{n}\leq K(u, v)+\e(|u_{n}-u|^{2^{*}_{s}}+|v_{n}-v|^{2^{*}_{s}})+C_{\e}(|u|^{2^{*}_{s}}+|v|^{2^{*}_{s}}).
$$
Now, we define
$$
w_{n, \e}:=(\xi_{n}-\e(|u_{n}-u|^{2^{*}_{s}}+|v_{n}-v|^{2^{*}_{s}}))^{+}.
$$
Let us observe that 
$$
0\leq w_{n, \e}\leq (M_{K}+C_{\e})(|u|^{2^{*}_{s}}+|v|^{2^{*}_{s}})\in L^{1}(\R^{N}, d\mu)
$$
where $M_{K}:=\max\{K(u, v): |u|^{2^{*}}+|v|^{2^{*}_{s}}=1\}$, and $w_{n, \e}\rightarrow 0$ a.e. in $\R^{N}$ in view of $(i)$.
Then, by using the Dominated Convergence Theorem, we can see that
$$
\lim_{n\rightarrow \infty} \int_{\R^{N}} w_{n, \e} d\mu=0.
$$
Recalling the definition of $\xi_{n}$ and by using $(ii)$, we can deduce that
$$
0\leq\limsup_{n\rightarrow \infty}\int_{\R^{N}} \xi_{n} d\mu\leq \e \limsup_{n\rightarrow \infty} \int_{\R^{N}} (|u_{n}-u|^{2^{*}_{s}}+|v_{n}-v|^{2^{*}_{s}}) d\mu\leq C\e.
$$
From the arbitrariness of $\e>0$ we deduce the thesis.
\end{proof}

\noindent
Now, we are ready to give the proof of Lemma \ref{CCL}.
\begin{proof}[Proof of Lemma \ref{CCL}]
In order to prove \eqref{47FS}, we aim to pass to the limit in the following relation which holds in view of Lemma \ref{lem8BL}:
\begin{align}\label{55fs}
&\int_{\R^{N}} |\psi|^{2^{*}_{s}} K(u_{n}, v_{n})\, dx \nonumber\\
&= \int_{\R^{N}} |\psi|^{2^{*}_{s}} K(u, v) \, dx+ \int_{\R^{N}} |\psi|^{2^{*}_{s}} K(u_{n}-u, v_{n}-v) \, dx + o_{n}(1), 
\end{align}
where $\psi\in C^{\infty}_{c}(\R^{N})$.\\
Set $\tilde{u}_{n}=u_{n}-u$, $\tilde{v}_{n}= v_{n}-v$. Then, by Theorem \ref{Sembedding}, we can see that $\tilde{u}_{n}, \tilde{v}_{n}\rightarrow 0$ in $L^{2}_{loc}(\R^{N})$ and a.e. on $\R^{N}$.\\
Fix $\psi\in C^{\infty}_{c}(\R^{N})$. By using the definition of $\widetilde{S}_{K}$ and $(K1)$, we get
\begin{align}\label{56fs}
&\left[ \int_{\R^{N}}  |\psi|^{2^{*}_{s}} K(u_{n}-u, v_{n}-v) \, dx\right]^{\frac{2}{2^{*}_{s}}} \nonumber \\
&= \left[ \int_{\R^{N}} |\psi|^{2^{*}_{s}} K(\tilde{u}_{n}, \tilde{v}_{n}) \, dx\right]^{\frac{2}{2^{*}_{s}}}\nonumber \\
&=\left[ \int_{\R^{N}} K(\psi \tilde{u}_{n}, \psi \tilde{v}_{n}) \, dx\right]^{\frac{2}{2^{*}_{s}}}\nonumber \\
&\leq \tilde{S}^{-1}_{K} \int_{\R^{N}} (|(-\Delta)^{\frac{s}{2}} (\psi\, \tilde{u}_{n})|^{2} +  |(-\Delta)^{\frac{s}{2}} (\psi \,\tilde{v}_{n})|^{2} )\, dx\nonumber \\
&=\tilde{S}^{-1}_{K} \left[\iint_{\R^{2N}} \frac{|(\psi\tilde{u}_{n})(x)-(\psi \tilde{u}_{n})(y)|^{2}}{|x-y|^{N+2s}}+\frac{|(\psi\tilde{v}_{n})(x)-(\psi \tilde{v}_{n})(y)|^{2}}{|x-y|^{N+2s}} \, dx dy\right].
\end{align}
Now, we observe that
\begin{align*}
&\iint_{\R^{2N}} \frac{|(\psi\tilde{u}_{n})(x)-(\psi \tilde{u}_{n})(y)|^{2}}{|x-y|^{N+2s}}+\frac{|(\psi\tilde{v}_{n})(x)-(\psi \tilde{v}_{n})(y)|^{2}}{|x-y|^{N+2s}} \, dx dy \nonumber\\
&\leq 2 \Biggl(\iint_{\R^{2N}}  |\psi(y)|^{2} \frac{|\tilde{u}_{n}(x)- \tilde{u}_{n}(y)|^{2}}{|x-y|^{N+2s}}  +|\tilde{u}_{n}(x)|^{2} \frac{|\psi(x)- \psi(y)|^{2}}{|x-y|^{N+2s}} \, dx dy \nonumber\\
&+ \iint_{\R^{2N}} |\psi(y)|^{2} \frac{|\tilde{v}_{n}(x)- \tilde{v}_{n}(x)|^{2}}{|x-y|^{N+2s}}  +|\tilde{v}_{n}(x)|^{2} \frac{|\psi(x)- \psi(y)|^{2}}{|x-y|^{N+2s}}\, dx dy\Biggr). 
\end{align*}
It is easy to show that
\begin{align*}
\iint_{\R^{2N}} \frac{|\psi(x)- \psi(y)|^{2}}{|x-y|^{N+2s}} (|\tilde{u}_{n}(x)|^{2}+ |\tilde{v}_{n}(x)|^{2})\, dxdy = o_{n}(1). 
\end{align*}
Indeed, arguing as in the proof of \eqref{NIOO}, if $\psi=1$ in $B_{1}$ and $\psi=0$ in $B_{2}^{c}$ we have
\begin{align*}
&\int_{\R^{N}} \,\int_{\R^{N}} \frac{|\psi(x)- \psi(y)|^{2}}{|x-y|^{N+2s}} |\tilde{u}_{n}(x)|^{2}\, dx dy \\
&=\int_{B_{2}} \int_{\R^{N}} \frac{|\psi(x)- \psi(y)|^{2}}{|x-y|^{N+2s}} |\tilde{u}_{n}(x)|^{2} dx dy+\int_{\R^{N}\setminus B_{2}} \int_{B_{2}} \frac{|\psi(x)- \psi(y)|^{2}}{|x-y|^{N+2s}} |\tilde{u}_{n}(x)|^{2}dx dy \\
&\leq C\int_{B_{K}}  |\tilde{u}_{n}(x)|^{2}\, dx +CK^{-N} \quad  \forall K>4,
\end{align*}
and taking the limit as $n\rightarrow \infty$ and then as $K\rightarrow \infty$ we get the thesis.\\
Therefore, if we assume that $|(-\Delta)^{\frac{s}{2}} \tilde{u}_{n}|^{2}\rightharpoonup \tilde{\mu}$, $|(-\Delta)^{\frac{s}{2}} \tilde{v}_{n}|^{2}\rightharpoonup \tilde{\sigma}$ and $|K(\tilde{u}_{n}, \tilde{v}_{n})|\rightharpoonup \tilde{\nu}$ in the sense of measures, from the above facts and by passing to the limit in \eqref{56fs} we have that
\begin{align*}
\left[ \int_{\R^{N}} |\psi|^{2^{*}_{s}} d\tilde{\nu}\right]^{\frac{1}{2^{*}_{s}}} \leq C \left[ \int_{\R^{N}} |\psi|^{2} (d\tilde{\mu} + d\tilde{\sigma})\right]^{\frac{1}{2^{*}_{s}}}, \, \mbox{ for all } \psi \in C^{\infty}_{c}(\R^{N}). 
\end{align*}
Then, by using Lemma $1.2$ in \cite{Lions}, there exist at most a countable set $I$, families $\{x_{i}\}_{i\in I}\subset \R^{N}$ and $\{\nu_{i}\}_{i\in I}\subset (0, \infty)$ such that
\begin{align}\label{57fs} 
\tilde{\nu} = \sum_{i\in I} \nu_{i} \delta_{x_{i}}.
\end{align}
In view of \eqref{55fs}, we deduce that $\nu= K(u, v) + \tilde{\nu}$ which together with \eqref{57fs}, implies that 
$$
\nu= K(u, v)+ \sum_{i\in I} \nu_{i} \delta_{x_{i}},
$$
that is \eqref{47FS} holds. \\
Now, we pass to prove \eqref{50FS}. Take $\psi_{\rho}= \eta(\frac{x-x_{i}}{\rho})$, where $\eta \in C^{\infty}_{c}(B_{1})$, $0\leq \eta \leq 1$ and $\eta(0)=1$. 
Then, recalling the definition of $\widetilde{S}_{K}$ and the inequality
\begin{align*}
(x+y)^{2}\leq x^{2} + C y^{2}, \, \mbox{ for all } x, y\geq 0,
\end{align*}
we can deduce that
\begin{align}\label{58fs}
&\tilde{S}_{K} \left[ \int_{\R^{N}} |\psi_{\rho}|^{2^{*}_{s}} K(u_{n}, v_{n})\, dx \right]^{\frac{2}{2^{*}_{s}}} \nonumber\\
&\leq \int_{\R^{N}} (|(-\Delta)^{\frac{s}{2}} (\psi_{\rho} \, u_{n})|^{2} + |(-\Delta)^{\frac{s}{2}} (\psi_{\rho} \, v_{n})|^{2} ) \, dx \nonumber \\
&\leq  C \Biggl( \iint_{\R^{2N}} |u_{n}(x)|^{2} \frac{|\psi_{\rho}(x)- \psi_{\rho}(y)|^{2}}{|x-y|^{N+2s}}+|v_{n}(x)|^{2} \frac{|\psi_{\rho}(x)- \psi_{\rho}(y)|^{2}}{|x-y|^{N+2s}} \, dxdy\Biggr) \nonumber\\
&+ \Biggl(\iint_{\R^{2N}} |\psi_{\rho}(y)|^{2} \frac{|u_{n}(x)- u_{n}(y)|^{2}}{|x-y|^{N+2s}}+  |\psi_{\rho}(y)|^{2} \frac{|v_{n}(x)- v_{n}(y)|^{2}}{|x-y|^{N+2s}} \,dxdy\Biggr).
\end{align}
Now, taking into account \eqref{47FS} and \eqref{46FS}, we have
\begin{align*}
\lim_{n\rightarrow \infty}\int_{\R^{N}} |\psi_{\rho}|^{2^{*}_{s}} K(u_{n}, v_{n})\, dx=\int_{B_{\rho}(x_{j})} |\psi_{\rho}|^{2^{*}_{s}} K(u, v)\, dx + \nu_{i}.
\end{align*}
Since $(K3)$ and $0\leq\psi_{\rho}\leq 1$ imply
\begin{equation*}
\left|\int_{B_{\rho}(x_{j})} |\psi_{\rho}|^{2^{*}_{s}} K(u, v)\, dx\right|\leq C\int_{B_{\rho}(x_{j})} (|u|^{2^{*}_{s}}+|v|^{2^{*}_{s}}) \, dx\rightarrow 0 \mbox{ as } \rho\rightarrow 0,
\end{equation*}
we can deduce that
\begin{align}\label{stef1}
\lim_{\rho\rightarrow 0}\lim_{n\rightarrow \infty}\int_{\R^{N}} |\psi_{\rho}|^{2^{*}_{s}} K(u_{n}, v_{n})\, dx= \nu_{i}.
\end{align}
On the other hand, \eqref{46FS} gives 
\begin{equation*}
\lim_{n\rightarrow \infty}\iint_{\R^{2N}} |\psi_{\rho}(y)|^{2} \frac{|u_{n}(x)- u_{n}(y)|^{2}}{|x-y|^{N+2s}}=\int_{\R^{N}} |\psi_{\rho}(y)|^{2} \, d\mu
\end{equation*}
and
\begin{equation*}
\lim_{n\rightarrow \infty}\iint_{\R^{2N}} |\psi_{\rho}(y)|^{2} \frac{|v_{n}(x)- v_{n}(y)|^{2}}{|x-y|^{N+2s}}=\int_{\R^{N}} |\psi_{\rho}(y)|^{2} \, d\sigma,
\end{equation*}
and arguing as in the proof of \eqref{NIOO} in Lemma \ref{lemma3.2-1}, we can see that
\begin{align}\label{stef3}
\lim_{\rho\rightarrow 0}\lim_{n\rightarrow \infty} &\iint_{\R^{2N}} |u_{n}(x)|^{2} \frac{|\psi_{\rho}(x)- \psi_{\rho}(y)|^{2}}{|x-y|^{N+2s}} \, dx dy=0 \nonumber\\
&=\lim_{\rho\rightarrow 0}\lim_{n\rightarrow \infty}\iint_{\R^{2N}} |v_{n}(x)|^{2} \frac{|\psi_{\rho}(x)- \psi_{\rho}(y)|^{2}}{|x-y|^{N+2s}} \, dx dy.
\end{align}
Then, putting together  \eqref{58fs}, \eqref{stef1} and \eqref{stef3}, we get
\begin{align*}
\tilde{S}_{K} \nu_{i}^{\frac{2}{2^{*}_{s}}} \leq \lim_{\rho\rightarrow 0} \mu(B_{\rho}(x_{i}))+ \lim_{\rho\rightarrow 0} \sigma(B_{\rho}(x_{i})). 
\end{align*}
Setting $\mu_{i}= \lim_{\rho\rightarrow 0} \mu(B_{\rho}(x_{i}))$ and $\sigma_{i}= \lim_{\rho\rightarrow 0} \sigma(B_{\rho}(x_{i}))$, we deduce that \eqref{50FS} holds. \\
Finally we can note that
\begin{align*}
\mu\geq \sum_{i\in I} \mu_{i} \delta_{x_{i}} \mbox{ and }  \sigma\geq \sum_{i\in I} \sigma_{i} \delta_{x_{i}}
\end{align*}
and that the weak convergences imply that $\mu \geq |(-\Delta)^{\frac{s}{2}}u|^{2}$ and $\sigma \geq |(-\Delta)^{\frac{s}{2}}v|^{2}$. Then, due to the fact that $|(-\Delta)^{\frac{s}{2}}u|^{2}$ and $|(-\Delta)^{\frac{s}{2}}v|^{2}$ are ortogonal to $\sum_{i\in I} \mu_{i} \delta_{x_{i}}$ and $\sum_{i\in I} \sigma_{i} \delta_{x_{i}}$ respectively, we can infer that \eqref{48FS} and \eqref{49FS} hold. 
\end{proof}

\noindent {\bf Acknowledgements.} 
The author would like to express his sincere thanks to the anonymous referee for his/her careful reading of the manuscript and valuable comments and suggestions.\\
The paper has been carried out under the auspices of the INdAM - GNAMPA Project 2017 titled: {\it Teoria e modelli per problemi non locali}.

\end{document}